\documentclass{amsart}

 \usepackage{amsmath}
\usepackage{amsfonts}
\usepackage{amssymb}
\usepackage[dvips]{graphicx}
\usepackage{color}
\usepackage{float} 
\usepackage{epsfig} 
\usepackage{url}
\newtheorem{thm}{Theorem}[section]
\newtheorem{theorem}[thm]{Theorem}

\newtheorem{lemma}[thm]{Lemma}

\newtheorem{proposition}[thm]{Proposition}

\newtheorem{corollary}[thm]{Corollary}

\newtheorem{remark}[thm]{Remark}

\newcommand{\beq}{\begin{equation}}
\newcommand{\eeq}{\end{equation}}
\newcommand{\beqa}{\begin{eqnarray}}
\newcommand{\eeqa}{\end{eqnarray}}
\newcommand{\beqas}{\begin{eqnarray*}}
\newcommand{\eeqas}{\end{eqnarray*}}
\newcommand{\bi}{\begin{itemize}}
\newcommand{\ei}{\end{itemize}}

\newcommand{\cX}{{\mathcal X}}

\newcommand{\argmin}{\mathrm{argmin}\,}

\def\Max{\mathop{\rm Max}}

\newcommand{\B}{{\mathcal B}}

\newcommand{\avr}{{\sf AV@R}}

\begin{document}

\title[Regularized SDDP]{Regularized Stochastic Dual Dynamic Programming for convex nonlinear optimization problems}

\maketitle

\vspace*{0.5cm}


\begin{center}
\begin{tabular}{ccc}
\begin{tabular}{c}
Vincent Guigues\\
School of Applied Mathematics, FGV\\
Praia de Botafogo, Rio de Janeiro, Brazil\\ 
{\tt vincent.guigues@fgv.br}
\end{tabular}&

&
\begin{tabular}{c}
Miguel Lejeune\\
The George Washington University\\
Washington, DC 20052, USA\\
{\tt mlejeune@gwu.edu}\\
\end{tabular}
\end{tabular}
\end{center}

\begin{center}
\begin{tabular}{c}
Wajdi Tekaya\\
Quant-Dev\\
Hammam Chatt, 1164, Tunisia,\\ 
{\tt tekaya.wajdi@gmail.com}
\end{tabular}
\end{center}

\date{}

\maketitle

\begin{abstract}
We define a regularized variant of the Dual Dynamic Programming algorithm called DDP-REG to solve 
nonlinear dynamic programming equations. We extend the algorithm to solve nonlinear stochastic dynamic programming equations. The 
corresponding algorithm, called SDDP-REG, can be seen as an extension of a regularization of the Stochastic Dual Dynamic Programming (SDDP) algorithm recently introduced which was studied for linear problems only and with less general prox-centers. We show the convergence of DDP-REG and 
SDDP-REG. We assess the performance of DDP-REG and SDDP-REG on portfolio models with direct transaction and market impact costs. In 
particular, we propose a risk-neutral portfolio selection model which can be cast as a  multistage stochastic second-order cone program. 
The formulation is motivated by the impact of market impact costs on  large portfolio rebalancing operations. Numerical simulations show that DDP-REG is much quicker than DDP on all problem instances considered (up to 184 times quicker than DDP) and that SDDP-REG is quicker on the instances of portfolio selection problems with market impact costs tested and much faster on the instance of risk-neutral multistage stochastic linear program implemented (8.2 times faster).\\
\end{abstract}
\vspace{-0.05in}



Multistage stochastic optimization problems are used to model many
real-life applications where a sequence of decisions has to be made,
subject to random costs and constraints arising from the observations
of a stochastic process, see \cite{powellbook, shadenrbook} for an overview on multistage stochastic programs. Solving such problems is challenging
and often requires some assumptions on the underlying stochastic process,
on the problem structure, and some sort of decomposition.
In this paper, we are interested in problems 
for which deterministic or 
stochastic dynamic programming equations can be written.
In this latter case, we will focus on situations where the underlying stochastic process is discrete interstage independent,
the number of stages is moderate to large, and the state vector is of small size.

Two popular solution methods to solve stochastic dynamic programming equations
are Approximate Dynamic Programming (ADP) and Stochastic Dual Dynamic Programming (SDDP) \cite{pereira}.
Several enhancements of SDDP have been
proposed such as the extension to interstage dependent stochastic processes \cite{morton,guiguescoap2013},
the introduction and analysis of risk-averse variants \cite{guiguesrom10,guiguesrom12,kozmikmorton,philpmatos,shapsddp,shaptekaya},
cut selection strategies \cite{pfeifferetalcuts,dpcuts0, guiguesejor2017, guiguesbandarra17}, and convergence proofs of the algorithm  and variants 
in  \cite{philpot,lecphilgirar12,guigues2015siopt,guiguesinexact2018,guiguesbandarra17}.
However, a known drawback of the method is its slow convergence rate. To cope with this difficulty, a regularized variant of SDDP
was recently proposed in \cite{powellasamov} for Multistage Stochastic Linear Programs (MSLPs). This variant consists in computing in the forward pass of SDDP the trial points using a 
Tikhonov regularization \cite{tikhonov}. More precisely, the objective is penalized with a quadratic term depending on a prox-center
updated at each iteration.
On the tests reported in \cite{powellasamov}, the regularized method converges faster than the classical SDDP method on risk-neutral instances of MSLPs.
On the basis of these encouraging numerical results, several natural questions arise:
\begin{itemize}
\item[a)] When specialized to deterministic problems, how does the regularized method behave? 
How to extend the
method when nonlinear objective and constraints are present and under which assumptions to ensure
the convergence of the method?
\item[b)] How can the regularized algorithm be extended to solve Multistage Stochastic NonLinear Problems (MSNLPs) and under which assumptions to ensure the convergence of the
algorithm to an optimal policy?
\item[c)] What other prox-centers and penalization schemes can be proposed? Find a MSLP for testing the new prox-centers and penalization schemes. Can we observe on this application
a faster convergence of the regularized method, as for the application considered in \cite{powellasamov}? 
\item[d)] Find a relevant application, modeled by a MSNLP, to test the regularized variant of SDDP.
\end{itemize}
The objective of this paper is to study items a)-d) above. Our findings on these topics are as follows:
\begin{itemize}
\item[a)] {\textbf{Regularized Dual Dynamic Programming.}} We propose a regularized variant
of Dual Dynamic Programming (DDP, the deterministic counterpart of SDDP) called DDP-REG, for nonlinear 
optimization problems. For DDP-REG, in Theorem \ref{convreddp}, we show the convergence of the
sequence of approximate first step optimal values to the optimal value of the problem
and that any accumulation point of the sequence of trial points is an optimal solution of the problem.
The same proof, with weaker assumptions can be used to show the convergence
of this regularized variant of DDP applied to linear problems.

We then consider instances of a portfolio problem with direct transaction costs with a large number of stages and compare the computational time required to solve these instances with DDP and DDP-REG. 
In all experiments, the computational time was drastically reduced using DDP-REG.
More precisely, we tested 6 different implementations of DDP-REG and for problems with $T=10, 50, 100, 150, 200, 250, 300$, and $350$ time periods,
the range (for these 6 implementations) of the reduction factor of the overall computational time with DDP-REG was respectively $[3.0,3.0], [13.8,17.3]$, $[22.3,33.5]$, 
$[37.1,65.0]$, $[46.6,76.7]$, $[80.0,114.3]$, $[71.5$, $171.6]$, and $[95.5,184.4]$.
\item[b)] {\textbf{SDDP-REG: Regularized SDDP.}}
We define a Regularized SDDP method for MSNLPs which samples in the backward pass
to compute cuts at trial points computed, as in \cite{powellasamov}, in a forward pass, penalizing the objective with a quadratic term depending on a prox-center. In Theorem \ref{convalg1}, we show the convergence of this algorithm.
More precisely, we show (i) the convergence
of the sequence of the optimal values of the approximate first stage problems and that 
(ii) any accumulation point of the sequence of decisions
can be used to define an optimal solution 
of the problem. It will turn out that (ii)
improves already known results for SDDP.
\item[c)] {\textbf{On prox-centers, penalization parameters, and on the performance of the regularization for MSLPs.}} We propose new prox-centers and penalization schemes and test 
them on risk-neutral and risk-averse instances of portfolio selection problems.
\item[d)] {\textbf{Portfolio Selection with Direct Transaction and Market Impact Costs.}}
The multistage optimization models studied in this paper are directly applicable in finance and in particular for the rebalancing of portfolios that incur transaction costs.
Transaction costs can have a major impact on the performance of an investment strategy (see, e.g., the survey \cite{CADE}). 
Two main types of transaction costs, implicit and explicit, can be distinguished. 

Explicit or {\it direct transaction costs} are directly observable, are directly charged to the investor, and are generally modelled as 
linear or piecewise linear. 
In reality, it is however not possible to trade arbitrary large quantities of securities at their current theoretical market price. 

Implicit or indirect costs, often called {\it market impact costs}, result from imperfect markets  due for example to market or liquidity restrictions (e.g., bid-ask spreads), 
depend on the order-book situation when the order is executed, and 
are not itemized explicitly, thereby making  it difficult for investors to recognize them. 
Yet, for large orders, they are typically much larger than the direct transaction costs. 
Market impact costs are equal to the difference between the transaction price and the (unperturbed) market price that would have prevailed if the trade had not occurred \cite{TORRE}.
Market impact costs are typically nonlinear (see, e.g., \cite{ALMG,ANDERSEN,GATHERAL,GRIN,TORRE}), and much more challenging to model than direct transaction costs.
Market impact costs are particularly important for large institutional investors, for which they can represent a major proportion of the total transaction costs \cite{MIBR,TORRE}. 
They can be viewed as an additional price for the immediate execution of large~trades.

There is a widespread interest in the modeling and analysis of market impact costs as they are (one of) the main reducible parts of the transaction costs \cite{MIBR}. 
In this study, we propose a series of dynamic - deterministic and stochastic (risk-neutral and risk-averse) - optimization models for portfolio optimization with direct transaction and 
market impact~costs.

We compare the computational time required to solve with SDDP-REG and SDDP instances of risk-neutral and risk-averse
portfolio problems with direct transaction costs.
We also compare the computational time required to solve with SDDP-REG and SDDP risk-neutral instances of
portfolio problems with market impact costs using real data and $T=48$ stages.
To our knowledge, no dynamic optimization problem for portfolio optimization with conic market impact costs has been proposed so far.
Also, we are not aware of other published numerical tests on the application of SDDP to a real-life application
modelled by a multistage stochastic second-order cone program with a large 
number of stages.
\end{itemize}

The paper is organized as follows. In section \ref{REDA}, we present a class of convex deterministic nonlinear optimization problems
for which dynamic programming equations can be written. We propose the variant DDP-REG of DDP
to solve these problems and show the convergence of the method in Theorem \ref{convreddp}.
In section \ref{SREFDA}, we introduce the type of stochastic nonlinear problems we are interested in
and propose SDDP-REG, a regularized decomposition algorithm to solve these problems.
In section \ref{convsreda}, we show in Theorem \ref{convalg1} the convergence of SDDP-REG. The portfolio selection models described in item d) above
are discussed in section \ref{FINM}. Finally, the last section \ref{numsim} presents the results of numerical simulations that illustrate our results.
We show that DDP-REG is much quicker than DDP on all problem instances considered (up to 184 times quicker than DDP) and that SDDP-REG is quicker on the instances of nonlinear stochastic
programs tested and much faster on the instance of risk-neutral multistage stochastic linear program implemented (8.2 times faster).\\

We use the following notation and terminology:
\par - The usual scalar product in $\mathbb{R}^n$ is denoted by $\langle x, y\rangle = x^T y$ for $x, y \in \mathbb{R}^n$.
The corresponding norm is $\|x\|=\|x\|_2=\sqrt{\langle x, x \rangle}$.
\par - $\mbox{ri}(A)$ is the relative interior of set $A$.
\par - $\mathbb{B}_n = \{x \in \mathbb{R}^n \, : \, \|x\| \leq 1\}$.
\par - dom($f$) is the domain of function $f$.
\par - $\mathcal{N}_A(x)$ is the normal cone to set $A$ at point $x$.
\par - $\avr_\alpha$ is the Average Value-at-Risk with confidence level $\alpha$, \cite{ury2}.
\par - $D(\mathcal{X})$ is the diameter of set $\mathcal{X}$.
\par - The notation $[A;B]$ represents the matrix $\left(\begin{array}{l}A\\B\end{array} \right)$.

\section{Regularized dual dynamic programming: Algorithm and convergence} \label{REDA}

\subsection{Problem formulation and assumptions}

Consider the problem
\vspace{-0.06in}
{\small
\begin{equation} \label{eq:DPP-prob}
\left\{
\begin{array}{l}
\min\; \sum_{t=1}^T f_t(x_{t-1},x_t) \\
x_t \in X_t( x_{t-1} ),\;\; \forall t=1,\ldots,T, 
\end{array}
\right. 
\end{equation}
}
where 
$X_t( x_{t-1} )  \subset \cX_t \subset \mathbb{R}^n$ is given by
$$
X_t( x_{t-1} )  = \{x_t \in \mathcal{X}_t \, : \, A_t x_t  + B_t x_{t-1}=b_t, g_t (x_{t-1}, x_t) \leq 0\},
$$
$f_t:\mathbb{R}^n \times \mathbb{R}^n \to \mathbb{R} \cup \{+\infty\}$ is a convex function, $g_t:\mathbb{R}^n \times \mathbb{R}^n  \to \mathbb{R}^p$,
and $x_0$ is given.\\
For this problem, we can write dynamic programming equations
defining recursively the functions $\mathcal{Q}_t$ on $\cX_{t-1}$ as
{\small
\begin{equation}\label{DDPeq}
\mathcal{Q}_{t}(x_{t-1}) := \min \left\{ f_{t}(x_{t-1}, x_t) + \mathcal{Q}_{t+1} ( x_t ) : x_t  \in X_{t}( x_{t-1} ) \right \}, \quad t=T,T-1,\ldots,1,
\end{equation}
}
with the convention that $\mathcal{Q}_{T+1} \equiv 0$. 
Clearly, $\mathcal{Q}_1 (x_0)$ is
the optimal value of \eqref{eq:DPP-prob}. More generally,
we have
\[
\mathcal{Q}_{t}(x_{t-1}) = \min \left\{ \sum_{j=t}^T f_j (x_{j-1},x_j) : x_j \in X_j ( x_{j-1} ), \, \forall  j=t,\ldots,T  \right\}.
\]
\par We make the following assumptions: setting $\mathcal{X}_0=\{x_0\}$ and
\begin{equation}\label{epsfattening}
\mathcal{X}_{t}^\varepsilon := \mathcal{X}_{t} + \varepsilon  \mathbb{B}_n 
\end{equation}
\par (H0) there exists $\varepsilon>0$ such that for $t=1,\ldots,T,$
\begin{itemize}
\item[(a)] $\mathcal{X}_{t} \subset \mathbb{R}^n$ is nonempty, convex, and compact;
\item[(b)] $f_t$ is proper, convex, and lower semicontinuous;
\item[(c)] Setting $g_t(x_{t-1}, x_t )=(g_{t, 1}(x_{t-1}, x_t ), \ldots, g_{t, p}(x_{t-1}, x_t ))$, for 
$i=1,\ldots,p$, the $i$-th component function
$g_{t, i}(x_{t-1}, x_t )$ is a convex lower semicontinuous function;
\item[(d)] 
$
\mathcal{X}_{t-1}^{\varepsilon} {\small{\times}}  \mathcal{X}_{t} \subset \mbox{dom}(f_t)
$
and for every $x_{t-1} \in \mathcal{X}_{t-1}^{\varepsilon}$, there exists 
$x_t \in \mathcal{X}_t$ such that $g_t(x_{t-1}, x_t ) \leq 0$ and $\displaystyle A_{t} x_{t} + B_{t} x_{t-1} = b_t$;
\item[(e)] if $t \geq 2$, there exists
$$
{\bar x}_{t}=({\bar x}_{t, t-1}, {\bar x}_{t, t}) \in 
\mathcal{X}_{t-1} \small{\times}  \mbox{ri}(\mathcal{X}_{t}) 
\cap \mbox{ri}(\{g_t \leq 0\})$$ such that $\bar x_{t, t} \in \mathcal{X}_t$,
$g_t(\bar x_{t, t-1}, \bar x_{t, t}) \leq 0$ and $A_{t} \bar x_{t, t} + B_{t} \bar x_{t, t-1} = b_t$.\\
\end{itemize}
The DDP algorithm solves \eqref{eq:DPP-prob} exploiting the convexity of recourse functions $\mathcal{Q}_t$:
\begin{lemma}\label{contqt} 
Consider recourse functions $\mathcal{Q}_t, t=2,\ldots,T+1$, given by 
\eqref{DDPeq}. 
Let Assumptions (H0)-(a), (H0)-(b), (H0)-(c), and (H0)-(d) hold. Then
for $t=2,\ldots,T+1$, $\mathcal{Q}_t$ is convex, finite on $\mathcal{X}_{t-1}^{\varepsilon}$, and Lipschitz continuous on $\mathcal{X}_{t-1}$.
\end{lemma}
\begin{proof}
See the proof of Proposition 3.1 in \cite{guigues2015siopt}.
\end{proof}
$\vphantom{}$\newline
The description of the subdifferential of $\mathcal{Q}_{t}$ given in the following proposition will be useful
for DDP, DDP-REG, and SDDP-REG: 
\begin{proposition} \label{pr:DPP-parta}
Lemma 2.1 in \cite{guigues2015siopt}. Let Asssumptions (H0) hold.
Let $x_t (x_{t-1})$ be an optimal solution of \eqref{DDPeq}.
Then for every $t=2,\ldots,T,$ for every $x_{t-1} \in \mathcal{X}_{t-1}$,
$s \in \partial \mathcal{Q}_{t}( x_{t-1} )$ if and only if
\small{
$$
\begin{array}{l}
(s,0) \in \partial f_t (x_{t-1}, x_t(x_{t-1}))
+  \Big\{[A_t^T; B_t^T ] \nu \;:\;\nu \in \mathbb{R}^q\Big\}\\
\hspace*{1.2cm}+ \Big\{\displaystyle \sum_{i \in I( x_{t-1}, x_t (x_{t-1})   )}\; \mu_i \partial g_{t, i}(x_{t-1}, x_t (x_{t-1}) )\;:\;\mu_i \geq 0 \Big\}+\{0\}\small{\times} \mathcal{N}_{\mathcal{X}_t}( x_t (x_{t-1}) )
\end{array}
$$
}
where $I(x_{t-1}, x_t (x_{t-1}) )=\Big\{i \in \{1,\ldots,p\} \;:\;g_{t, i}(x_{t-1}, x_t (x_{t-1})) =0\Big\}.$
\end{proposition}
\begin{proof}
See \cite{guigues2015siopt}.
\end{proof}

\subsection{Dual Dynamic Programming}

We first recall DDP method to solve \eqref{DDPeq}. It uses approximations $\mathcal{Q}_t^{k}$ of $\mathcal{Q}_t$.
At iteration $k$, let functions 
$\mathcal{Q}^{k}_t : \cX_{t-1} \to \mathbb{R}$ such that
\beq \label{eq:DPP-lower-functs}
\mathcal{Q}^{k}_{T+1} = \mathcal{Q}_{T+1}, \quad \mathcal{Q}^{k}_t \le \mathcal{Q}_t \quad t=2,3,\ldots,T,
\eeq
be given and define for $t=1,2,\ldots,T$ the function ${\underline{\mathcal{Q}}}_t^k : \cX_{t-1} \to \mathbb{R}$ as
\[
{\underline{\mathcal{Q}}}_t^k  (x_{t-1}) = \min \left\{ f_t(x_{t-1}, x_t ) + \mathcal{Q}^{k}_{t+1} ( x_t ) : x_t \in X_t ( x_{t-1} ) \right \} \quad \forall x_{t-1} \in \mathcal{X}_{t-1}.
\]
Clearly,  \eqref{eq:DPP-lower-functs} implies that: 
${\underline{\mathcal{Q}}}_T^k = \mathcal{Q}_T, \quad {\underline{\mathcal{Q}}}_t^k \le \mathcal{Q}_t \quad t=1,2,\ldots,T-1.$
It is assumed that the functions $\mathcal{Q}^{k}_t$ can be evaluated at any point $x_{t-1} \in \cX_{t-1}$.
The DDP algorithm works as follows:\\\\
\begin{minipage}[h]{5.1 in}
{\bf DDP (Dual Dynamic Programming).}
\begin{itemize}
\item[Step 1)] {\textbf{Initialization.}} Let $\mathcal{Q}^{0}_t : \cX_{t-1} \to \mathbb{R} \cup \{-\infty\},\,t=2,\ldots,T+1$, 
satisfying \eqref{eq:DPP-lower-functs} be given.  Set $k=1$.
\item[Step 2)] {\textbf{Forward pass.}} Setting $x_0^{k} = x_0$, for $t=1,2,\ldots,T$,
compute
\begin{align}
x_t^{k} \in \argmin \left\{ f_t(x_{t-1}^{k}, x_t ) + \mathcal{Q}^{k-1}_{t+1}( x_t ) : x_t \in X_t( x_{t-1}^{k}  )    \right \}.
\end{align}
\item[Step 3)] {\textbf{Backward pass.}}  Define $\mathcal{Q}_{T+1}^k \equiv 0$. For $t=T, T-1, \ldots,2$,
solve 
\begin{align}
{\underline{\mathcal{Q}}}_t^k (x_{t-1}^{k}) = \min \left\{ f_t(x_{t-1}^{k}, x_t ) + \mathcal{Q}^{k}_{t+1}( x_t ) : x_t \in X_t( x_{t-1}^{k}  )    \right \},
\end{align}
using Proposition \ref{pr:DPP-parta} take a subgradient $\beta_t^k$ of ${\underline{\mathcal{Q}}}_t^k (  \cdot )$ at $x_{t-1}^{k}$, and store the new cut
\[
\mathcal{C}_t^k ( x_{t-1} ) := {\underline{\mathcal{Q}}}_t^k ( x_{t-1}^k ) + \langle \beta_t^k , x_{t-1} - x_{t-1}^{k} \rangle
\]
for $\mathcal{Q}_t$,
making up the new approximation
$\mathcal{Q}_{t}^k = \max\{ \mathcal{Q}^{k-1}_{t}, \mathcal{C}_t^k  \}$.

\item[Step 4)] Do $k \leftarrow k+1$ and go to Step 2).
\end{itemize}

\noindent
\end{minipage}

\subsection{Regularized Dual Dynamic Programming}

For the regularized DDP to be presented in this section, we still define
\[
{\underline{\mathcal{Q}}}_t^k  (x_{t-1}) = \min \left\{ F_t^{k}(x_{t-1}, x_t )  \,   : x_t \in X_t ( x_{t-1} ) \right \} \quad \forall x_{t-1} \in \mathcal{X}_{t-1}, \; \text{where}
\]
\begin{equation}\label{defFtk}
F_t^{k}(x_{t-1}, x_t ) =   f_t(x_{t-1}, x_t ) + \mathcal{Q}^{k}_{t+1} ( x_t ).
\end{equation}
However, since the function $\mathcal{Q}^{k}_{t+1}$ computed by regularized DDP is different from the function $\mathcal{Q}^{k}_{t+1}$ computed by DDP, the functions ${\underline{\mathcal{Q}}}_t^k$ obtained with respectively regularized DDP and DDP are different. The regularized DDP algorithm is given~below:\\\\
\begin{minipage}[h]{5.1in}
{\bf Regularized DDP (DDP-REG).}
\begin{itemize}
\item[Step 1)] {\textbf{Initialization.}} Let $\mathcal{Q}^{0}_t : \cX_{t-1} \to \mathbb{R} \cup \{-\infty\},\,t=2,\ldots,T+1$, 
satisfying \eqref{eq:DPP-lower-functs} be given. Set $k=1$.
\item[Step 2)] {\textbf{Forward pass.}} Setting $x_0^{k} = x_0$, for $t=1,2,\ldots,T$,
compute
\begin{equation}\label{forreg}
x_t^{k} \in \argmin     \left\{ {\bar F}_t^{k-1}(x_{t-1}^k, x_t, x_t^{P, k} ) \; : x_t \in X_t( x_{t-1}^{k}  )     \right \},
\end{equation}
where the prox-center $x_t^{P, k}$ is any point in $\mathcal{X}_t$  and
where ${\bar F}_t^{k-1}:\mathcal{X}_{t-1} \small{\times} \mathcal{X}_{t} \small{\times} \mathcal{X}_{t} \rightarrow \mathbb{R}$ is 
$$
{\bar F}_t^{k-1}(x_{t-1}, x_t , x_t^P ) = f_t(x_{t-1} , x_t  ) + \mathcal{Q}^{k-1}_{t+1}( x_t ) + \lambda_{t, k} \|x_t - x_t^P \|^2
$$
for some exogenous nonnegative penalization $\lambda_{t, k}$ with $\lambda_{t, k}=0$ if $t=T$ or $k=1$.
\item[Step 3)] {\textbf{Backward pass.}} Define $\mathcal{Q}_{T+1}^k \equiv 0$.
For $t=T, T-1, \ldots,2$,
solve 
\begin{align}
{\underline{\mathcal{Q}}}_t^k (x_{t-1}^{k}) = \min \left\{ f_t(x_{t-1}^{k}, x_t ) + \mathcal{Q}^{k}_{t+1}( x_t ) : x_t \in X_t( x_{t-1}^{k}  )    \right \},
\end{align}
using Proposition \ref{pr:DPP-parta} take a subgradient $\beta_t^k$ of ${\underline{\mathcal{Q}}}_t^k (  \cdot )$ at $x_{t-1}^{k}$, and store the new cut
\[
\mathcal{C}_t^k ( x_{t-1} ) := {\underline{\mathcal{Q}}}_t^k ( x_{t-1}^k ) + \langle  \beta_t^k ,  x_{t-1} - x_{t-1}^{k} \rangle
\]
for $\mathcal{Q}_t$,
making up the new approximation
$\mathcal{Q}_{t}^k = \max\{ \mathcal{Q}^{k-1}_{t}, \mathcal{C}_t^k  \}$.
\item[Step 4)] Do $k \leftarrow k+1$ and go to Step 2).
\end{itemize}
\noindent
\end{minipage}

Observe that the backward passes of the regularized and non-regularized DDP are the same.
The algorithms differ from the way the trial points are computed: for regularized DDP a proximal term
is added to the objective function of each period to avoid moving too far from the prox-center.

\subsection{Convergence analysis}\label{convreddp:sec}

The following lemma will be useful to analyze the convergence of regularized DDP:
\begin{lemma}\label{contqtk}
Let Assumptions (H0) hold. Then the functions $\mathcal{Q}_t^k, t=2,\ldots,T+1, k \geq 1,$ generated by DDP-REG are
Lipschitz continuous on 
$\mathcal{X}_{t-1}^{\varepsilon}$, satisfy $\mathcal{Q}_t^k \leq \mathcal{Q}_t$,
and ${\underline{\mathcal{Q}}}_t^k ( x_{t-1}^k )$ and $\beta_t^k$ are bounded for all $t \geq 2, k \geq 1$.
\end{lemma}
\begin{proof} It suffices to follow the proof of Lemma 3.2 in \cite{guigues2015siopt}.
\end{proof}
$\vphantom{}$\\
We have for DDP-REG the following convergence theorem which is a special case of 
Theorem \ref{convalg1} shown in section \ref{convsreda} (obtained considering deterministic processes 
$\xi_t$).
\begin{theorem} \label{convreddp} Consider the sequences of decisions $x_t^k$ and approximate recourse functions $\mathcal{Q}_t^k$
generated by DDP-REG. Let Assumptions (H0) hold and
assume that for $t=1,\ldots,T-1$, we have $\lim_{k \rightarrow +\infty} \lambda_{t, k}=0$ and $\lambda_{T, k}=0$
for every $k \geq 1$.
Then we have $\mathcal{Q}_{T+1}( x_T^k)=\mathcal{Q}_{T+1}^k( x_T^k)$,
\begin{equation}\label{equalqT}
\mathcal{Q}_{T}( x_{T-1}^k)= \mathcal{Q}_{T}^k( x_{T-1}^k ) = {\underline{\mathcal{Q}}}_{T}^k( x_{T-1}^k), 
\end{equation}
and for $t=2,\ldots,T-1$,
$$
H(t): \;\displaystyle \lim_{k \to +\infty} \mathcal{Q}_t( x_{t-1}^k ) - \mathcal{Q}_{t}^{k}(x_{t-1}^{k}) =\displaystyle \lim_{k \to +\infty} \mathcal{Q}_t( x_{t-1}^k ) -{\underline{\mathcal{Q}}}_{t}^{k}(x_{t-1}^{k}) = 0.
$$
Also, (i) $\lim_{k \rightarrow +\infty} {\underline{\mathcal{Q}}}_1^k  (x_{0}) = \lim_{k \rightarrow +\infty} {\bar F}_1^{k-1}(x_{0}, x_1^k, x_1^{P, k} ) =\mathcal{Q}_1( x_0 )$, the optimal value of \eqref{eq:DPP-prob},
and (ii) any accumulation point $(x_1^*, \ldots,x_T^*)$ of the sequence $(x_1^k, \ldots, x_T^k)_k$ is an optimal solution of \eqref{eq:DPP-prob}.
\end{theorem}

If convergence of DDP-REG holds for any sequence $(x_t^{P,k})_{k \geq 2}$ of prox-centers in $\mathcal{X}_t$
and of penalty parameters $\lambda_{t, k}$ converging to zero for every $t$,
the performance of the method depends on how these sequences are chosen.
DDP is obtained taking $\lambda_{t, k}=0$ for every $t, k$. 

For all numerical experiments of section  \ref{simdet}, DDP-REG was much faster than DDP.
Some natural candidates for $\lambda_{t, k}$ and $x_t^{P,k}$, used in our numerical tests,
are the following:
\begin{itemize}
\item Weighted average of previous values: $x_t^{P,k}=\frac{1}{\Gamma_{t, k}} \sum_{j=1}^{k-1} \gamma_{t, k, j} x_t^j$
with $\gamma_{t, k, j}$ nonnegative weights and $\Gamma_{t, k}=\sum_{j=1}^{k-1} \gamma_{t, k, j}$.
Note that $x_t^{P,k} \in \mathcal{X}_t$ because all $x_t^j$ are in the convex set $\mathcal{X}_t$.
Special cases include the average of previous values  $x_t^{P,k}=\frac{1}{k-1} \sum_{j=1}^{k-1} x_t^j$
and the last trial point $x_t^{P,k}= x_t^{k-1}$ for $t<T$, $k \geq 2$. 
\item  $\lambda_{t, k} = \rho_t^k$ where $0 < \rho_t < 1$ or 
$\lambda_{t, k} = \frac{1}{k^2}$ for  $t<T$, $k \geq 2$.
\end{itemize}

If for a given stage $t$, $\mathcal{X}_t$ is a polytope and we do not have the nonlinear constraints
given by constraint functions $g_t$ (i.e., the constraints for this stage are linear), then the conclusions of 
Lemmas \ref{contqt}, \ref{contqtk},
and Theorem \ref{convreddp} hold under weaker assumptions. 
More precisely, for such stages $t$, we assume (H0)-a),
(H0)-(b), and instead of (H0)-(d), (H0)-(e), the weaker assumption (H0)-(c'):\\
\par (H0)-(c') There exists $\varepsilon>0$ such that:
\begin{itemize}
\item[(c').1)] $
\mathcal{X}_{t-1}^{\varepsilon}{\small{\times}} \mathcal{X}_{t}  \subset \mbox{dom} \; f_t;$
\item[(c').2)] for every
$x_{t-1} \in \mathcal{X}_{t-1}$, the set $X_t(x_{t-1})$ is nonempty.
\end{itemize}

\section{Regularized Stochastic Dual Dynamic programming} \label{SREFDA}

\subsection{Problem formulation and assumptions}

Consider a stochastic process $(\xi_t)$ where
$\xi_t$ is a discrete random vector with finite support 
containing in particular as components the entries
in $(b_t, A_t, B_t )$ in a given order where $b_t$ are random vectors
and $A_t, B_t$ are random matrices.

Let $\mathcal{F}_{t}$ denote the sigma-algebra $\sigma(\xi_1,\ldots,\xi_t)$, 
$\mathcal{Z}_t$ be the set of $\mathcal{F}_t$-measurable functions, and
$\mathbb{E}_{|\mathcal{F}_{t-1}}: \mathcal{Z}_{t} \rightarrow \mathcal{Z}_{t-1}$ be the conditional expectation at $t$.

With this notation, we are interested in solving problems of form
\begin{equation} \label{pbinit0}
\begin{array}{ll}
\displaystyle \inf_{x_1 \in X_1(x_0, \xi_{1})}&f_{1}(x_0, x_{1}, \xi_1) +  \mathbb{E}_{|\mathcal{F}_1}\left( \displaystyle \inf_{x_2 \in X_2(x_{1}, \xi_{2})}\;f_{2}(x_{1}, x_2, \xi_2 ) + \ldots \right.\\
&+\mathbb{E}_{|\mathcal{F}_{T-2}}\left( \displaystyle \inf_{x_{T-1} \in X_{T-1}(x_{T-2},\,\xi_{T-1})}\;f_{T-1}(x_{T-2}, x_{T-1}, \xi_{T-1}) \right.\\
&+ \mathbb{E}_{|\mathcal{F}_{T-1}}\left. \left. \left( \displaystyle \inf_{x_{T} \in X_{T}(x_{T-1},\,\xi_{T})}\;f_{T}(x_{T-1}, x_T ,\xi_{T})  \right) \right) \ldots \right)
\end{array}
\end{equation}
for some functions $f_t$ taking values in $\mathbb{R}\cup \{+\infty\}$, 
where $x_0$ is given and where
$$
X_{t}(x_{t-1},\, \xi_{t}) = \Big\{x_t \in \mathcal{X}_t \;:\;g_t(x_{t-1}, x_t, \xi_t) \leq 0,\;\;A_t x_t +B_t x_{t-1} = b_t\Big\}
$$
for some vector-valued function $g_t$ and some nonempty compact convex set $\mathcal{X}_t \subset \mathbb{R}^n$.

We make the following assumption on $(\xi_t)$:
\begin{itemize}
\item[(H1)] $(\xi_t)$ 
is interstage independent and
for $t=2,\ldots,T$, $\xi_t$ is a random vector taking values in $\mathbb{R}^K$ with discrete distribution and
finite support $\Theta_t=\{\xi_{t, 1}, \ldots, \xi_{t, M}\}$ while $\xi_1$ is deterministic.\\
\end{itemize}
To alleviate notation and without loss of generality, we have assumed that the number $M$ of possible realizations
of $\xi_t$, the size $K$ of $\xi_t$, and dimension $n$ of $x_t$ do not depend on $t$.

Under Assumption (H1), $\mathbb{E}_{|\mathcal{F}_{t-1}}$ coincides with its unconditional counterpart
$\mathbb{E}_{t}$ where $\mathbb{E}_{t}$ is the expectation computed with respect to the distribution of $\xi_t$.
To ease notation, we will drop the index $t$ in $\mathbb{E}_t$.
As a result, for problem \eqref{pbinit0}, we can write the following dynamic programming equations:
we set $\mathcal{Q}_{T+1} \equiv 0$ and for
$t=2,\ldots,T$, define 
{\small
\begin{equation}\label{definitionQt}
\mathcal{Q}_{t}(x_{t-1})=\mathbb{E}\Big(\mathfrak{Q}_{t}(x_{t-1}, \xi_t)\Big)
\end{equation}
}
with
\begin{equation} \label{defmathfrak}
\begin{array}{lll}
\mathfrak{Q}_{t}(x_{t-1}, \xi_t)&=&
\left\{
\begin{array}{l}
\displaystyle \inf_{x_t}\;F_t(x_{t-1}, x_{t}, \xi_t):=f_t(x_{t-1}, x_t, \xi_{t}) + \mathcal{Q}_{t+1}(x_{t})\\
x_t \in X_t (x_{t-1}, \xi_t ). 
\end{array}
\right.
\end{array}
\end{equation}
Problem \eqref{pbinit0} can then be written
\begin{equation} \label{firststagepb}
\left\{
\begin{array}{l}
\displaystyle \inf_{x_1} \; F_1(x_0, x_{1}, \xi_1):=f_1(x_0, x_{1}, \xi_{1}) + \mathcal{Q}_{2}(x_{1})\\
x_1 \in X_1 (x_{0}, \xi_1 )=\{
x_1 \in \mathcal{X}_1 : g_1(x_0, x_1, \xi_1) \leq 0, A_{1} x_{1} + B_1 x_{0}  = b_1 \},
\end{array}
\right.
\end{equation}
with optimal value denoted by $\mathcal{Q}_{1}(x_{0})=\mathfrak{Q}_{1}(x_{0}, \xi_1)$.

Recalling definition \eqref{epsfattening} of the $\varepsilon$-fattening of a set, we make the following Assumption (H2): setting $\mathcal{X}_0=\{x_0\}$, there exists 
$\varepsilon>0$ such that for $t=1,\ldots,T$:
\begin{itemize}
\item[1)] $\mathcal{X}_{t} \subset \mathbb{R}^n$ is nonempty, convex, and compact.
\item[2)] For every $j=1,\ldots,M$, the function
$f_t(\cdot, \cdot,\xi_{t, j})$ is proper, convex, and lower semicontinuous.
\item[3)] For every $j=1,\ldots,M$, each component of the function $g_t(\cdot,\cdot, \xi_{t, j})$ is a
convex lower semicontinuous function.
\item[4)] we have
\begin{itemize}
\item[4.1)] for every $j=1,\ldots,M$,
$   \mathcal{X}_{t-1}^{\varepsilon}{\small{\times}} \mathcal{X}_{t}  \subset \mbox{dom} \; f_t(\cdot, \cdot, \xi_{t, j});   $
\item[4.2)] for every $j=1,\ldots,M$, for every $x_{t-1} \in \mathcal{X}_{t-1}^{\varepsilon}$, the set $X_t(x_{t-1}, \xi_{t, j})$ is nonempty.
\end{itemize}
\item[5)] If $t \geq 2$, for every $j=1,\ldots,M$, there exists
$$
{\bar x}_{t, j}=({\bar x}_{t, j, t-1}, {\bar x}_{t, j, t}) \in \mathcal{X}_{t-1} \small{\times}\mbox{ri}(\mathcal{X}_t)
\cap \mbox{ri}(\{g_t(\cdot,\cdot,\xi_{t, j})\leq 0\})$$ such that $\bar x_{t, j, t} \in X_t(\bar x_{t, j, t-1}, \xi_{t, j})$.\\
\end{itemize}

The following proposition, proved in \cite{guigues2015siopt}, shows that Assumption (H2) guarantees that for $t=2,\ldots,T$, recourse function $\mathcal{Q}_t$ is convex and Lipschitz continuous on the set $\mathcal{X}_{t-1}^{\hat{\varepsilon}}$ for every $0<{\hat{\varepsilon}}<\varepsilon$.
SDDP-REG and its convergence analysis are based on this proposition. 
\begin{proposition}\label{convexityrec} 
Under Assumption (H2), for $t=2,\ldots,T+1$, for every $0<{\hat{\varepsilon}}<\varepsilon$, the recourse function $\mathcal{Q}_t$ is convex, finite and  Lipschitz continuous on
$\mathcal{X}_{t-1}^{\hat{\varepsilon}}$.
\end{proposition}
\begin{proof} We refer to the proof of Proposition 3.1 in \cite{guigues2015siopt} where similar value functions are considered.
\end{proof}
$\vphantom{}$\\Assumption (H2) will also be used to derive explicit formulas for the cuts to be built for recourse functions $\mathcal{Q}_t$ in SDDP-REG applied to the nonlinear problems we are interested in.

\subsection{Algorithm} \label{sredalg}

Recalling Assumption (H1), the distribution of $(\xi_2, \ldots, \xi_T)$ is discrete and the $M^{T-1}$ possible realizations of $(\xi_2, \ldots, \xi_T)$ can be organized in a finite scenario tree with the root node $n_0$ associated to a stage $0$ (with decision $x_0$ taken at that node) having one child node $n_1$ associated to the first stage (with $\xi_1$ deterministic). In this section, we describe SDDP-REG algorithm, a regularization of SDDP
which can be seen as an extension of the regularization proposed
in \cite{powellasamov} (which applied to linear problems) to nonlinear problems.

To describe this algorithm, we need some notation: $\mathcal{N}$ is the set of nodes, {\tt{Nodes}}$(t)$ is the set of nodes of the scenario tree for stage $t$ and for a node $n$ of the tree, we denote by: 
\begin{itemize}
\item $C(n)$ the set of its children nodes (the empty set for the leaves);
\item $x_n$ a decision taken at that node;
\item $p_n$ the transition probability from the parent node of $n$ to $n$;
\item $\xi_n$ the realization of process $(\xi_t)$ at node $n$\footnote{Note that
to simplify notation, the same notation $\xi_{\tt{Index}}$ is used to denote
the realization of the process at node {\tt{Index}} of the scenario tree and the value of the process $(\xi_t)$
for stage {\tt{Index}}. The context will allow us to know which concept is being referred to.
In particular, letters $n$ and $m$ will only be used to refer to nodes while $t$ will be used to refer to stages.}:
for a node $n$ of stage $t$, this realization $\xi_n$ contains in particular the realizations
$b_n$ of $b_t$, $A_{n}$ of $A_{t}$, and $B_{n}$ of $B_{t}$;
\item $\xi_{[n]}$ is the history of the realizations of process $(\xi_t)$ from the first stage node $n_1$ to node $n$:
for a node $n$ of stage $t$, the $i$-th component of $\xi_{[n]}$ is $\xi_{\mathcal{P}^{t-i}(n)}$ for $i=1,\ldots, t$,
where $\mathcal{P}:\mathcal{N} \rightarrow \mathcal{N}$ is the function 
associating to a node its parent node (the empty set for the root node).
\end{itemize}

At iteration $k$ of the algorithm, trial points $x_n^k$ are computed
for a set of sampled nodes  $n$ of the scenario tree replacing recourse functions 
$\mathcal{Q}_{t+1}$ by the approximations $\mathcal{Q}_{t+1}^{k-1}$ available at the beginning of this iteration and 
penalizing the objective with a quadratic term with prox-center $x_t^{P, k}$ for all the nodes
of stage $t$. The nodes selected at iteration $k$ are denoted
$(n_1^k, n_2^k, \ldots, n_T^k)$ 
(with $n_1^k=n_1$, and for $t \geq 2$, $n_t^k$ a node of stage $t$, child of node $n_{t-1}^k$) 
and correspond to a sample $({\tilde \xi}_1^k, {\tilde \xi}_2^k,\ldots, {\tilde \xi}_T^k)$
of $(\xi_1, \xi_2,\ldots, \xi_T)$. For $t=2,\ldots,T$, a cut  
\begin{equation}\label{eqcutctk}
\mathcal{C}_t^k( x_{t-1} ) = \theta_t^{k} + \langle \beta_t^{k}, x_{t-1}-x_{n_{t-1}^{k}}^{k} \rangle
\end{equation}
is computed for $\mathcal{Q}_t$ at $x_{n_{t-1}^{k}}^{k}$ (see the algorithm below for details).
To alleviate notation, we will write $x_{t-1}^k := x_{n_{t-1}^{k}}^{k}$.

Gathering the cuts computed until iteration $k$, we get at the end of this iteration for $\mathcal{Q}_t$
the polyhedral lower approximations $\mathcal{Q}_{t}^{k},\;t=2,\ldots,T+1$, given by
$$
\begin{array}{lll}
\mathcal{Q}_{t}^{k}(x_{t-1}) & = &\displaystyle \max_{0 \leq \ell \leq k}\;\mathcal{C}_t^{\ell}( x_{t-1} ).
\end{array}
$$ 
To describe and analyze the algorithm, it is convenient to introduce the function ${\underline{\mathfrak{Q}}}_t^k: \mathcal{X}_{t-1} \small{\times} \Theta_t \rightarrow \mathbb{R}$ given by
\begin{equation}\label{defmatfrak}
{\underline{\mathfrak{Q}}}_t^k( x_n , \xi_m ) = \left\{
\begin{array}{l}
\inf_{x_m} \;F_t^k(x_n, x_m, \xi_m)\\ 
x_m \in X_t(x_n, \xi_m)
\end{array} 
\right.
\end{equation}
where
\begin{equation}
F_t^k(x_n, x_m, \xi_m) = f_t(x_n, x_m, \xi_m) + \mathcal{Q}_{t+1}^k( x_m ). 
\end{equation}
The SDDP-REG algorithm is given below:\\
\par {\bf SDDP-REG (Regularized SDDP).}\\
\begin{itemize}
\item[Step 1)] {\textbf{Initialization.}} Let $\mathcal{Q}^{0}_t: \cX_{t-1} \to \mathbb{R} \cup \{-\infty\},\,t=2,\ldots,T$, 
satisfying $\mathcal{Q}^{0}_t \leq \mathcal{Q}_t$ be given and $\mathcal{Q}^{0}_{T+1} \equiv 0$. Set 
$\mathcal{C}_t^{0}=\mathcal{Q}^{0}_t, t=2,\ldots,T+1$, $k=1$.
\item[Step 2)] {\textbf{Forward pass.}} \\
Sample a scenario $(\xi_1,\tilde \xi_2^k,\ldots,\tilde \xi_T^k)$ from the distribution
of $\xi^k=(\xi_1,\xi_2^k,\ldots,\xi_T^k) \sim (\xi_1,\xi_2,\ldots,\xi_T)$ corresponding to a set of nodes $(n_1,n_2^k,\ldots,n_T^k)$.\\
{\textbf{For }}$t=1,\ldots,T$,\\
\hspace*{0.8cm}Find an optimal solution $x_t^k$ of
\begin{equation} \label{defxtkj}
\left\{
\begin{array}{l}
\displaystyle \inf_{x_t} \; {\bar F}_t^{k-1}(x_t, x_{t-1}^k , x_t^{P, k}, {\tilde \xi}_t^k) \\
x_t \in X_t( x_{t-1}^k, {\tilde \xi}_t^k ),
\end{array}
\right.
\end{equation}
\hspace*{2.4cm}where $x_{0}^k = x_0$, $x_t^{P, k}$ is any point in $\mathcal{X}_t$ and where\\
\hspace*{2.4cm}${\bar F}_t^{k-1}$ is the function given by
\begin{equation} \label{objforward}
 \begin{array}{llll}
\hspace*{1.5cm}&{\bar F}_t^{k-1}(x_t, x_{t-1} , x_t^{P}, \xi_t) &=& f_t(x_{t-1} , x_t, \xi_{t}) + \mathcal{Q}_{t+1}^{k-1}(x_t ) \\
&&&+ \lambda_{t, k}\|x_t - x_t^{P}\|^2,
\end{array}
\end{equation}
\hspace*{2.4cm}with $\lambda_{t, k}=0$ if $t=T$ or $k=1$.\\
{\textbf{End For}}
\item[Step 3)] {\textbf{Backward pass.}}\\
Set $\theta_{T+1}^k=0$ and $\beta_{T+1}^k=0$.\\
{\textbf{For }}$t=T,\ldots,2$,\\
\hspace*{0.8cm}{\textbf{For }}every child node $m$ of $n=n_{t-1}^k$ solve
$$
{\underline{\mathfrak{Q}}}_t^k( x_n , \xi_m ) = \left\{
\begin{array}{l}
\inf_{x_m} \;F_t^k(x_n , x_m, \xi_m)\\ 
x_m \in X_t(x_n, \xi_m)
\end{array}
\right.
$$
\hspace*{1.8cm}and compute, using Proposition \ref{pr:DPP-parta}, a subgradient\\
\hspace*{1.8cm}$\pi_m^k \in \partial {\underline{\mathfrak{Q}}}_t^k(\cdot, \xi_m)$ at $x_n^k$.\footnote{Note that the proposition can be applied because Assumption (H2) holds and thus the assumptions of the proposition
are satisfied for value function ${\underline{\mathfrak{Q}}}_t^k( \cdot , \xi_m )$.}\\
\hspace*{0.8cm}{\textbf{End For}}\\
\hspace*{0.8cm}The new cut $\mathcal{C}_t^k$ is obtained  computing
\begin{equation}\label{formulathetak}
\theta_t^k=\sum_{m \in C(n)} p_m {\underline{\mathfrak{Q}}}_t^k( x_n^k , \xi_m ),\;\;\beta_t^k=\sum_{m \in C(n)} p_m  \pi_m^k.
\end{equation}
{\textbf{End For}}
\item[Step 4)] Do $k \leftarrow k+1$ and go to Step 2).
\end{itemize}

\subsection{On the prox-centers and penalizations}\label{proxpenaliz}

Though $x_t^{P,k}$ are now random variables, the remarks of section \ref{convreddp} on the choice of the prox-centers
for DDP-REG still apply for SDDP-REG.
Indeed, convergence of SDDP-REG holds for any sequence $(x_t^{P,k})_{k \geq 2}$ of prox-centers in $\mathcal{X}_t$
and of penalty parameters $\lambda_{t, k}$ converging to zero for every $t$,
but the performance of the method depends on how these sequences are chosen.
The following choices for $\lambda_{t, k}$ and $x_t^{P,k}$ will be used in our numerical tests of SDDP-REG:
\begin{itemize}
\item Weighted average of previous values: $x_t^{P,k}=\frac{1}{\Gamma_{t, k}} \sum_{j=1}^{k-1} \gamma_{t, k, j} x_t^j$
with $\gamma_{t, k, j}$ nonnegative weights and $\Gamma_{t, k}=\sum_{j=1}^{k-1} \gamma_{t, k, j}$. As a special case,
$x_t^{P,k}=x_t^{k-1}$
was used in \cite{powellasamov}.
\item  $\lambda_{t, k} = \rho_t^k$ where $0 < \rho_t < 1$ (used in \cite{powellasamov} with $\rho_t$ constant) or 
$\lambda_{t, k} = \frac{1}{k^2}$ for  $t<T$, $k \geq 2$.
\end{itemize}

\section{Convergence analysis of SDDP-REG}\label{convsreda}

The approximate recourse functions $\mathcal{Q}_{t}^{k-1}$ available 
at the end of iteration $k-1$ of SDDP-REG define a policy allowing us to compute
decisions $x_n^k$ for every node $n$ of the scenario tree with the following
loops:
\rule{\linewidth}{1pt}
{\textbf{Simulation of SDDP-REG in the end of iteration $k-1$.}}\\
\rule{\linewidth}{1pt}
{\textbf{For }}$t=1,\ldots,T$,\\
\hspace*{0.8cm}{\textbf{For }}every node $n$ of stage $t-1$,\\
\hspace*{1.6cm}{\textbf{For }}every child node $m$ of node $n$, compute an optimal solution $x_m^k$ of
\begin{equation} \label{defxtkj}
\left\{
\begin{array}{l}
\displaystyle \inf_{x_m} \; {\bar F}_t^{k-1}(x_n^k, x_m , x_t^{P, k}, \xi_m) \\
x_m \in X_t( x_n^k, \xi_m ),
\end{array}
\right.
\end{equation}
\hspace*{2.4cm}where $x_0^k=x_0$.\\
\hspace*{1.6cm}{\textbf{End For}}\\
\hspace*{0.8cm}{\textbf{End For}}\\
{\textbf{End For}}\\
\rule{\linewidth}{1pt}

We will assume that the sampling procedure in SDDP-REG satisfies the following property:\\
\begin{itemize}
\item[(H3)] for every $j=1,\ldots,M$, for every $t=2,\ldots,T$, and for every $k \in \mathbb{N}^*$,
$
\mathbb{P}(\xi_t^k = \xi_{t, j})=\mathbb{P}(\xi_t = \xi_{t, j})>0.
$ For every 
$t=2,\ldots,T$, and $k \geq 1$,
$$
\xi_t^k \mbox{ is independent on }\sigma(\xi_2^1,\ldots,\xi_T^1,\ldots,\xi_{2}^{k-1}, \ldots, \xi_{T}^{k-1},\xi_2^k,\ldots,\xi_{t-1}^k).
$$
\end{itemize}

The following lemma will be useful in the sequel:

\begin{lemma}\label{lipqtk}
Consider the sequences $\mathcal{Q}_t^k, \theta_t^k$, and $\beta_t^k$ generated by SDDP-REG.
Under Assumptions (H1), (H2), then almost surely, for $t=2,\ldots,T+1$, the following holds:
\begin{itemize}
\item[(a)] $\mathcal{Q}_t^k$ is convex with
$\mathcal{Q}_t^{k} \leq \mathcal{Q}_{t}$ on 
$\mathcal{X}_{t-1}^{\varepsilon}$ for all $k \geq 1$;
\item[(b)] the sequences
$(\theta_t^k)_{k \geq 1}$ and  $(\beta_t^k)_{k \geq 1}$ are bounded;
\item[(c)] for $k \geq 1$,
$\mathcal{Q}_t^k$ is Lipschitz continuous on 
$\mathcal{X}_{t-1}^{\varepsilon}$.
\end{itemize}
\end{lemma}
\begin{proof} The proof is similar to the proof of Lemma 3.2 in \cite{guigues2015siopt}.\footnote{In \cite{guigues2015siopt}
a forward, instead of a forward-backward algorithm, is considered. In this setting, finiteness of coefficients $\theta_t^k$ and $\beta_t^k$ is not guaranteed
for the first iterations (for instance $(\theta_t^1)_t$ are $-\infty$) but the proof is similar.}
We give the main steps of the proof which is by backward induction on $t$ starting with $t=T+1$ where the statement holds by definition
of $\mathcal{Q}_{T+1}$. Assuming for $t \in \{2,\ldots,T\}$ that $\mathcal{Q}_{t+1}^k$ is Lipschitz
continuous on $\mathcal{X}_{t}^{\varepsilon}$ with $\mathcal{Q}_{t+1}^k \leq \mathcal{Q}_{t+1}$,
then setting $n=n_{t-1}^k$, for every $m \in C(n)$ we have $\mathfrak{Q}_t(\cdot, \xi_m) \geq {\underline{\mathfrak{Q}}}_t^k (\cdot, \xi_m)$ which gives
$$
\begin{array}{lll}
 \mathcal{Q}_t(x_{t-1}) &= & \displaystyle \sum_{m \in C(n)} p_m \mathfrak{Q}_t( x_{t-1}, \xi_m )\\
 & \geq & \displaystyle \sum_{m \in C(n)} p_m {\underline{\mathfrak{Q}}}_t^k ( x_{t-1}, \xi_m )\\
 & \geq & \displaystyle \sum_{m \in C(n)} p_m \Big( {\underline{\mathfrak{Q}}}_t^k ( x_{n}^k, \xi_m )  + \langle \pi_m^k , x_{t-1}  - x_n^k \rangle \Big) = \mathcal{C}_t^k(x_{t-1} )   ,                        
\end{array}
$$
where for the last inequality, we have used Proposition \ref{pr:DPP-parta} which can be applied since Assumption (H2)-5) holds.
Therefore, $\mathcal{C}_t^k$ defines a valid cut for $\mathcal{Q}_t$ and $\mathcal{Q}_t \geq \mathcal{Q}_t^k$.
Assumptions (H2)-1)-4) and finiteness of $\mathcal{Q}_t$ on $\mathcal{X}_{t-1}^{\varepsilon}$ imply that 
${\underline{\mathfrak{Q}}}_t^k ( x_{n}^k, \xi_m )$ and $\pi_m^k$  are bounded for every $m \in C(n)$,
and allow us to obtain a uniform upper bound
on $\beta_t^k$, i.e., a Lipschitz constant valid for all functions $\mathcal{Q}_t^k, t=2,\ldots,T+1, k\geq 1$.\hfill
\end{proof}

$\vphantom{}$\\Theorem \ref{convalg1}  shows the convergence of the sequence ${\underline{\mathfrak{Q}}}_{1}^k(x_0, \xi_1)$ 
to the optimal value $\mathcal{Q}_{1}(x_0)$ of \eqref{pbinit0} 
and that 
any accumulation point of the sequence $((x_n^k)_{n \in \mathcal{N}})_{k \geq 1}$ 
can be used to define an optimal solution 
of \eqref{pbinit0}.
\begin{theorem}[Convergence analysis of SDDP-REG]\label{convalg1}
Consider the sequences of stochastic decisions $x_n^k$ and of recourse functions $\mathcal{Q}_ t^k$
generated by SDDP-REG to solve dynamic programming equations  \eqref{definitionQt}-\eqref{defmathfrak}-\eqref{firststagepb}.
Let Assumptions (H1), (H2), and (H3) hold and assume that $\lambda_{T,k}=0$ and that 
for every $t=1,\ldots,T-1$, we have
$\lim_{k \rightarrow +\infty} \lambda_{t, k}=0$. Then
\begin{itemize}
\item[(i)] almost surely, for $t=2,\ldots,T+1$, the following holds:
$$
\mathcal{H}(t): \;\;\;\forall n \in {\tt{Nodes}}(t-1), \;\; \displaystyle \lim_{k \rightarrow +\infty} \mathcal{Q}_{t}(x_{n}^{k})-\mathcal{Q}_{t}^{k}(x_{n}^{k} )=0.
$$
\item[(ii)]
Almost surely, the limit of the sequence
$( {\bar F}_1^{k-1}(x_0, x_{n_1}^k , x_1^{P, k}, \xi_1) )_k$ of the approximate first stage optimal values
and of the sequence
$({\underline{\mathfrak{Q}}}_1^{k}(x_{0}, \xi_1))_k$
is the optimal value 
$\mathcal{Q}_{1}(x_0)$ of \eqref{pbinit0}. 
Also, let $(x_n^*)_{n \in \mathcal{N}}$ be any accumulation point
of the sequence $((x_n^k)_{n \in \mathcal{N}})_{k \geq 1}$. 
Now define $x_1,\ldots,x_T$ with $x_t: \mathcal{Z}_t \rightarrow \mathbb{R}^n$ by
$x_t( \xi_1, \ldots, \xi_t  )=x_{m}^*$ where $m$ is given by $\xi_{[m]}=(\xi_1,\ldots,\xi_t)$.
Then $(x_1,\ldots,x_T)$ is an optimal solution to \eqref{pbinit0}.
\end{itemize}
\end{theorem}
\begin{proof} See the Appendix.\hfill
\end{proof}
$\vphantom{}$\\
\begin{remark}[Extension of SDDP-REG to risk-averse nonlinear problems.]\label{importantremark2}
Using \cite{guigues2015siopt}, SDDP-REG algorithm can be extended to nested risk-averse formulations of risk-averse
multistage stochastic nonlinear programs of form
\begin{equation} \label{sredara}
\begin{array}{ll}
\displaystyle \inf_{x_1 \in X_1(x_0, \xi_{1})}&f_{1}(x_{1},\,\xi_1) +  \rho_{2|\mathcal{F}_1}\left( \displaystyle \inf_{x_2 \in X_2(x_{1}, \xi_{2})}\;f_{2}(x_{1:2}, \xi_2 ) + \ldots \right.\\
&+\rho_{T-1|\mathcal{F}_{T-2}}\left( \displaystyle \inf_{x_{T-1} \in X_{T-1}(x_{T-2},\,\xi_{T-1})}\;f_{T-1}(x_{1:T-1}, \xi_{T-1}) \right.\\
&+ \rho_{T|\mathcal{F}_{T-1}}\left. \left. \left( \displaystyle \inf_{x_{T} \in X_{T}(x_{T-1},\,\xi_{T})}\;f_{T}(x_{1:T}, \xi_{T})  \right) \right) \ldots \right),
\end{array}
\end{equation}
where $\rho_{t+1|\mathcal{F}_t}: \mathcal{Z}_{t+1} \rightarrow \mathcal{Z}_{t}$ is a coherent and law invariant conditional risk measure.
The convergence proof of this variant of SDDP-REG can be easily obtained
combining the convergence proof of risk-averse decomposition methods from \cite{guigues2015siopt}
with the convergence proof of Theorem \ref{convalg1}.
\end{remark}
$\vphantom{}$\\
Similary to DDP-REG, if for a given stage $t$, $\mathcal{X}_t$ is a polytope and we do not have the nonlinear constraints
given by constraint functions $g_t$ (i.e., the constraints for this stage are linear), then the conclusions of 
Proposition \ref{convexityrec}, Lemma \ref{lipqtk},
and Theorem \ref{convalg1} hold under weaker assumptions. 
More precisely, for such stages $t$, we assume (H2)-1),
(H2)-2), and instead of (H2)-4), (H2)-5), the weaker assumption (H2)-3'): there exists $\varepsilon>0$ such that 3.1') for every $j=1,\ldots,M$,
$
\mathcal{X}_{t-1}^{\varepsilon}{\small{\times}} \mathcal{X}_{t}  \subset \mbox{dom} \; f_t\Big(\cdot, \cdot, \xi_{t, j}\Big);
$
and  3.2') for every $j=1,\ldots,M$, for every
$x_{t-1} \in \mathcal{X}_{t-1}$, the set $X_t(x_{t-1}, \xi_{t, j})$ is nonempty.

\section{Multistage portfolio optimization models with direct transaction and market impact costs} \label{FINM}

\subsection{Multistage portfolio selection models with direct transaction costs}\label{linearmodels}

This section presents risk-neutral and risk-averse multistage portfolio optimization models with direct transaction costs over a discretized horizon
of $T$ stages. We model the direct transaction costs incurred by 
selling and purchasing securities as being proportional to the amount of the transaction {\cite{FILE}.

Let $n$ be the number of risky assets and asset $n+1$ be cash. 
Next $x_t^i$ is the dollar value of asset $i=1,\ldots,n+1$ at the end of stage $t=1,\ldots,T$, 
$\xi_t^i$ is the return of asset $i$ at $t$, 
$y_t^i$ is the amount of asset $i$ sold at the end of $t$, 
$z_t^i$ is the amount of asset $i$ bought at the end of $t$, 
$\eta_i > 0$ and $\nu_i > 0$ are respectively the proportional selling and purchasing transaction costs.
Each component $x_0^i,i=1,\ldots,n+1$, of $x_0$ is a known parameter. 
The expression $\sum_{i=1}^{n+1} \xi_{1}^i x_{0}^i$ is the budget available at the start of the investment planning horizon. 
The notation $u_i$ is a parameter defining the maximal amount that can be invested in each financial security $i$. To allow for a direct application of SDDP-REG to solve the problem, we will assume that 
returns $\xi_t$ are interstage independent. Simple modifications of SDDP-REG could be 
used to deal with returns following generalized linear models with finite memory (as in \cite{guiguescoap2013}) and with Markov chain approximations of the returns (as in \cite{mogrundt}).

For $t=1,\ldots,T$, given a portfolio $x_{t-1}=(x_{t-1}^1,\ldots,x_{t-1}^n, x_{t-1}^{n+1})$ and $\xi_t$, we define the set $X_t(x_{t-1}, \xi_t)$
as the set of $(x_t, y_t, z_t) \in \mathbb{R}^{n+1} \small{\times} \mathbb{R}^{n} \small{\times} \mathbb{R}^{n}$ satisfying
\begin{equation}
x_t^{n+1}=\xi_{t}^{n+1} x_{t-1}^{n+1}   +\sum\limits_{i=1}^{n} \left((1-\eta_i)y_t^i-  (1+\nu_i)z_t^i\right), \label{p1_3}
\end{equation}
and for $i=1,\ldots,n$,
\begin{subequations}\label{MI-CONS0}
\begin{align}
  x^i_{t}&= \xi_{t}^i x_{t-1}^i-y_t^i+z_t^i,  \label{p1_2}\\
	y_t^i &\leq  \xi_{t}^{i} x_{t-1}^{i},  \label{p1_4}\\
  x_t^i &\leq u_i  \sum\limits_{i=1}^{n+1} \xi_{t}^i x_{t-1}^i,  \label{p1_5}\\
	x_t^i, y_t^i, z_t^i & \ge  0.  \label{p1_8}
\end{align}
\end{subequations}
Constraints \eqref{p1_2} define the amount of security $i$ held at each stage $t$ and take into account the proportional transaction costs.
Constraints \eqref{p1_3} are the cash flow balance constraints and define how much cash is available at each stage.
Constraints \eqref{p1_4} preclude selling an amount larger than the one held. 
Constraints \eqref{p1_5} do not allow the position in security $i$ at time $t$ to exceed a specified limit $u_i$, while 
\eqref{p1_8} prevents short-selling and enforces the non-negativity of the amounts purchased and~sold.\\ 

\par {\textbf{Risk-neutral model.}} With this notation, the dynamic programming equations of a risk-neutral portfolio model
of form \eqref{definitionQt}, \eqref{defmathfrak}, \eqref{firststagepb} can be written\footnote{It is indeed immediately seen that \eqref{eq1dp}-\eqref{eq2dp} is of form \eqref{definitionQt}, \eqref{defmathfrak}, \eqref{firststagepb},
writing the maximization problems as minimization problems and introducing the extended state $s_t=(x_t, y_t, z_t)$.}:
for $t=T$, setting $\mathcal{Q}_{T+1}( x_T ) = \mathbb{E}[ \sum\limits_{i=1}^{n+1} \xi_{T+1}^i x_{T}^i ]$
we solve the problem
\begin{equation}\label{eq1dp}
\mathfrak{Q}_T \left( x_{T-1}, \xi_T \right)=
\left\{
\begin{array}{l}
\text{Max} \; \mathcal{Q}_{T+1}( x_T )  \\
(x_T, y_T, z_T) \in X_T(x_{T-1}, \xi_T),
\end{array}
\right.
\end{equation}
while at stage $t=T-1,\dots,1$, we solve
\begin{equation}\label{eq2dp}
\mathfrak{Q}_t\left( x_{t-1}, \xi_t  \right)=
\left\{
\begin{array}{l}
\text{Max}  \;  Q_{t+1}\left( x_{t} \right) \\
(x_t, y_t, z_t) \in X_t(x_{t-1}, \xi_t) , 
\end{array}
\right.
\end{equation}
where for $t=2,\ldots,T$, $\mathcal{Q}_t(x_{t-1})=\mathbb{E}[ \mathfrak{Q}_t\left( x_{t-1}, \xi_t  \right) ]$.
With this model, we maximize the expected return of the portfolio taking into account the transaction costs,
non-negativity constraints, and bounds imposed on the different securities.\\

\par {\textbf{Risk-averse model.}} As we recall from the previous section,
SDDP-REG can be easily  extended to solve risk-averse problems of form \eqref{sredara}.
We can therefore define a nested risk-averse counterpart of the risk-neutral portfolio problem
we have just introduced and solve it with SDDP-REG. This model is obtained replacing the expectation in the risk-neutral portfolio
problem above by the (unconditional, due to Assumption (H1)) risk measure $\rho_t:\mathcal{Z}_t \rightarrow \mathbb{R}$ given by
$$
\rho_t \left[ Z \right] = (1-\kappa_t) \mathbb{E}\left[ Z \right] + \kappa_t \avr_{\alpha_t} \left[ Z \right] ,
$$
where $\kappa_t \in (0,1)$, ${\alpha_t} \in (0,1)$ is the confidence level of the Average Value-at-Risk, and 
$\rho_t$ is computed with respect to the distribution of $\xi_t$.
Therefore, a risk-averse portfolio problem with direct transaction costs is written as follows:
at stage $T$,  setting $\mathcal{Q}_{T+1}( x_T ) = \rho_{T+1} [ \sum\limits_{i=1}^{n+1} \xi_{T+1}^i x_{T}^i ]$,
we~solve 
\begin{equation}\label{eq1dpra}
\mathfrak{Q}_T \left( x_{T-1}, \xi_T \right)=
\left\{
\begin{array}{l}
\text{Max} \; \mathcal{Q}_{T+1}( x_T )  \\
(x_T, y_T, z_T) \in X_T(x_{T-1}, \xi_T),
\end{array}
\right.
\end{equation}
while at stage $t=T-1,\dots,1$, we solve
\begin{equation}\label{eq2dpra}
\mathfrak{Q}_t\left( x_{t-1}, \xi_t  \right)=
\left\{
\begin{array}{l}
\text{Max}  \; Q_{t+1}\left( x_{t} \right) \\
(x_t,y_t, z_t) \in X_t(x_{t-1}, \xi_t),
\end{array}
\right.
\end{equation}
where for $t=2,\ldots,T$, $\mathcal{Q}_t(x_{t-1})=\rho_t [ \mathfrak{Q}_t\left( x_{t-1}, \xi_t  \right) ]$.

\subsection{Conic quadratic models for multistage portfolio selection with market impact costs}\label{marketimpcostmodel}


Due to market imperfections, securities can seldom be traded at their current theoretical market price, which leads to additional costs, called market impact costs. 
If the trade is very large and involves the purchase (resp., selling) of a security, the price of the share may rise (resp., drop) between the placement of the trade and the completion of 
its execution \cite{MORO}. As more of a security is bought or sold, the proportional cost increases due to the scarcity effect. 
Market impact costs are particularly important for large institutional investors, for which they represent a major proportion of the total transaction costs \cite{MIBR,TORRE}. 
Often, large trading orders are not executed at once, but are instead split into a sequence of smaller orders executed within a given time window. Taken individually, these small orders exert little or no pressure on the market \cite{ZAKA}, which can curb market impact costs. The downside is that the execution of the entire trade order is postponed, which may lead to a loss in opportunities caused by (unfavorable) changes in market prices. 

The change in a security price is impacted by the size of the transaction and is often modelled as a concave monotonically increasing function of the trade size \cite{ATH}. 
In that vein, Lillo et al. \cite{LILLO} and Gabaix et al. \cite{GGPS} model market impact costs as a concave power law function of the transaction size. Bouchaud et al. \cite{BGPW} use a logarithmic function of the transaction size and assert that the market impact is temporary and decays as a power law. 
Moazeni et al. \cite{MOAZENI} propose linear market impact costs and evaluate the sensitivity of optimal execution strategies with respect to errors in the estimation of the parameters.
Mitchell and Braun \cite{MIBR} study the standard portfolio selection problem in which they incorporate convex transaction costs, including market impact costs, incurred 
when rebalancing the portfolio. They rescale the budget available after paying transaction costs, which results into a fractional problem that can be reformulated as a convex one. 
Frino et al. \cite{FBWL} approximate impact costs with a linear regression based on quantized transaction sizes, while Zagst and Kalin \cite{ZAKA} use a piecewise linear function. 
Loeb \cite{LOEB} shows that market impact costs are a function of the square root of the amount traded. 
Similarly, Torre \cite{TORRE} models the price change as proportional to the square root of the order size. This led to the so-called {\it square-formula} which defines the market impact costs as proportional to the square root of the ratio of the number of shares traded to the average daily trading volume of the security \cite{GATHERAL}. 
The square-root formula is widely used in the financial practice \cite{GATHERAL} to provide a pre-trade estimate of market impact costs and is preconized by Andersen et al. \cite{ANDERSEN} as well as  by Grinold and Kahn \cite{GRKA}. 
The latter observe that this approach is consistent with the trading rule-of-thumb according to which it costs roughly one day's volatility to trade one day's volume. 
In Barra's Market Impact Model Handbook \cite{TORRE}, it is showed  that the square-root formula fits transaction cost data remarkably well. 
An empirical study conducted by Almgren et al. \cite{ATH} advocates to set the price change as proportional to a 3/5 power law of block~size. 

In this study, the modeling of the market impact costs is based on the square-formula. 
More precisely, we follow the approach proposed by Grinold and Kahn \cite{GRKA} and Andersen \cite{ANDERSEN}, and model the market impact costs as proportional to a 3/2 power law of the transacted amount (see \eqref{TMIC}).  

Let $\alpha_t^i$ be the volume of security $i$ in the considered transaction and $\gamma_t^i$ be the overall market volume for security $i$ at $t$. Additionally, $g_t^i$ is the monetary value of asset $i$ transacted at $t$.
The market impact costs for asset $i$ are defined as:
\begin{equation}
\label{MIC}
\theta^i_t \sqrt{\frac{\alpha^i_t}{\gamma^i_t}} \approx m^i_t \sqrt{\left| g^i_t \right|} \; ,
\end{equation}
where $\theta^i_t$ and $m^i_t$ are non-negative parameters that must be estimated.  The market impact costs capture the fact that the price of an asset increases or decreases if one buys or sells very many shares of this asset. 
The total market impact costs depend on both the cost per unit $m^i_t$ and the square root of the amount traded $g^i_t$ (which is aligned with the empirical tests reported in \cite{LOEB}): 
\vspace{-0.035in}
\begin{equation}
\label{TMIC}
g^i_t m^i_t  \sqrt{\left| g^i_t \right|} \; .
\end{equation}
For $t=1,\ldots,T$, given a portfolio $x_{t-1}=(x_{t-1}^1,\ldots,x_{t-1}^n, x_{t-1}^{n+1})$ and $\xi_t$, we define 
\vspace{-0.05in}
{\small 
\begin{equation}\label{setmimpact}
X_t^{\mathcal{M} \mathcal{I}}(x_{t-1}, \xi_t)=
\left\{
\begin{array}{ll}
(x_t, y_t, z_t, q_t, g_t) \in \mathbb{R}_{+}^{n+1} \small{\times} \mathbb{R}_{+}^{n} \small{\times} \mathbb{R}_{+}^{n} \small{\times} \mathbb{R}_{+}^{n} \small{\times} \mathbb{R}_{+}^{n}:&\\
\eqref{p1_2}-\eqref{p1_5}, i=1,\ldots,n,&\\
x_t^{n+1}= \xi_{t}^{n+1} x_{t-1}^{n+1}   +\sum\limits_{i=1}^{n}(y_t^i - z_t^i - q_t^i), &(a)\\
g_t^i = y_t^i+z_t^i,\;  i=1,\ldots,n,&(b)\\
g_t^i m_t^i   \sqrt{g_t^i} \leq q_t^i ,\;  i=1,\ldots,n & (c)  
\end{array}
\right\}.
\end{equation}
}
Constraints \eqref{setmimpact}-(a)  define how much cash is held at each period and take into account the market impact costs. 
Constraints \eqref{setmimpact}-(b)  define the total amount $g_t^i$ of security $i$ traded at time $t$. 
The nonlinear constraints \eqref{setmimpact}-(c) follow from \eqref{TMIC} and permit to define the total market impact costs $q_t^i$ incurred for security $i$ at time $t$. Same as in the previous section, we assume that returns $(\xi_t)$ are interstage independent.

The risk-neutral multistage portfolio optimization problem with market impact costs writes as follows: 
for $t=T$, setting $\mathcal{Q}_{T+1}( x_T ) = \mathbb{E}[ \sum\limits_{i=1}^{n+1} \xi_{T+1}^i x_{T}^i ]$
we solve the problem
\begin{equation}\label{eq1dpm}
\mathfrak{Q}_T \left( x_{T-1}, \xi_T \right)=
\left\{
\begin{array}{l}
\text{Max} \; \mathcal{Q}_{T+1}( x_T )  \\
(x_T, y_T, z_T, q_T, g_T) \in X_T^{\mathcal{M} \mathcal{I}}(x_{T-1}, \xi_T),
\end{array}
\right.
\end{equation}
while at stage $t=T-1,\dots,1$, we solve
\begin{equation}\label{eq2dpm}
\mathfrak{Q}_t\left( x_{t-1}, \xi_t  \right)=
\left\{
\begin{array}{l}
\text{Max}  \; \mathbb{E}\left[ Q_{t+1}\left( x_{t} \right) \right] \\
(x_t, y_t, z_t, q_t, g_t) \in X_t^{\mathcal{M} \mathcal{I}}(x_{t-1}, \xi_t) , 
\end{array}
\right.
\end{equation}
where for $t=2,\ldots,T$, $\mathcal{Q}_t(x_{t-1})=\mathbb{E}[ \mathfrak{Q}_t\left( x_{t-1}, \xi_t  \right) ]$.

It is easy to see that the left-hand side of the constraints \eqref{setmimpact}-(c) are convex functions, which 
implies that Assumptions (H2)-3) are satisfied and SDDP-REG can be applied to solve the portfolio problem
under consideration. For implementation purposes, it is convenient to rewrite constraints 
\eqref{setmimpact}-(c) as a conic quadratic constraint:
\begin{theorem}
\label{TH1}
For $t=1,\ldots,T$, the convex feasible sets 
\[
\mathcal{S}_t = 
\left\{
\begin{array}{l}
(g_t, q_t)=(g_t^1,\ldots,g_t^n,q_t^1,\ldots,q_t^n) \in \mathbb{R}_+^n \small{\times}\mathbb{R}_+^n : \\
g_t^i m_t^i \sqrt{g_t^i} \leq q_t^i,\,i=1,\ldots,n  
\end{array}
\right\}
\]
can be equivalently represented with the rotated second-order constraints \eqref{E1} and the linear constraints
\eqref{E3}-\eqref{E22}:
\begin{subequations}\label{MI-CONS}
  \begin{align}
 &(\ell_t^i)^2 \leq 2 s_t^i \frac{q_t^i} {m_t^i},  \quad   
(w_t^i)^2 \leq 2 v_t^i r_t^i, \;        & i=1,\ldots,n,  \label{E1} \\    
 &\ell_t^i = v_t^i, \quad s_t^i = w_t^i,       & i=1,\ldots,n, \label{E3} \\     
&r_t^i = 0.125, \quad 
s_t^i, v_t^i \geq 0,        & i=1,\ldots,n,   \label{E6}\\
&-g_t^i \leq \ell_t^i, \quad 
g_t^i  \leq \ell_t^i   & i=1,\ldots,n. \label{E22}
 \end{align}
\end{subequations}
\end{theorem}
\begin{proof} This representation is proved in \cite{ANDERSEN}.
\end{proof}

$\vphantom{}$\\For $t=1,\ldots,T$, given $x_{t-1} \in \mathbb{R}^{n+1}$ and $\xi_t$, denoting 
\begin{equation}\label{setmimpactmodified}
{\mathbb{X}}_t^{\mathcal{M} \mathcal{I}}(x_{t-1}, \xi_t)=
\left\{
\begin{array}{l}
(x_t, y_t, z_t, q_t, g_t,\ell_t, s_t, v_t, w_t) \in \\
\mathbb{R}_{+}^{n+1} \small{\times} \mathbb{R}_{+}^{n} \small{\times} \mathbb{R}_{+}^{n} \small{\times} \mathbb{R}_{+}^{n} \small{\times} \mathbb{R}_{+}^{n}\small{\times} \mathbb{R}_{+}^{n} \small{\times} \mathbb{R}_{+}^{n} \small{\times} \mathbb{R}_{+}^{n} \small{\times} \mathbb{R}_{+}^{n}:\\
\eqref{p1_2}-\eqref{p1_5}, i=1,\ldots,n,\\
x_t^{n+1}=\xi_{t}^{n+1} x_{t-1}^{n+1}   +\sum\limits_{i=1}^{n}(y_t^i - z_t^i - q_t^i), \\
g_t^i = y_t^i+z_t^i,\;  i=1,\ldots,n,\\
\eqref{E1}-\eqref{E22} ,\;  i=1,\ldots,n  
\end{array}
\right\}
\end{equation}
and using Theorem \ref{TH1}, our portfolio optimization problem with market impact costs 
\eqref{eq1dpm}-\eqref{eq2dpm} can be rewritten substituting in 
\eqref{eq1dpm}-\eqref{eq2dpm} for $t=1,\ldots,T$, the constraints
$(x_t, y_t, z_t, q_t, g_t) \in X_t^{\mathcal{M} \mathcal{I}}(x_{t-1}, \xi_t)$
by 
$$
(x_t, y_t, z_t, q_t, g_t, \ell_t, s_t, v_t, w_t) \in {\mathbb{X}}_t^{\mathcal{M} \mathcal{I}}(x_{t-1}, \xi_t).
$$
This formulation of the portfolio problem can be solved
using SDDP-REG with all subproblems of the forward and backward passes being conic quadratic
optimization problems.

\if{
\begin{corollary}
\label{COR1}
The nonlinear programming problem \eqref{MI} can be reformulated as the second-order cone programming problem including $2n$ conic constraints
\begin{align}
  \Max \ &\eqref{F1} \notag \\
  {\rm s.t.} \; & \eqref{p1_2} \ ; \ \eqref{p1_4}-\eqref{p1_8} \ ; \ \eqref{F3}-\eqref{F4} \ ; \ \eqref{E1}-\eqref{E22} \notag \ .
 \end{align}
\end{corollary}

Assuming stagewise independence.
The problem at stage $T$ is:
\begin{subequations}\label{MI-T}
\begin{align}
Q_T\left( x_{T-1} \right)=\text{Max} \ & \mathbb{E}\left[\sum\limits_{i=1}^{n+1}(1+r^i_{T})   x^i_{T} \right] \label{OBJT}&\\
{\rm s.t.} \ &x^i_{t}=(1+r_{t-1}^i) x_{T-1}^i-y_t^i+z_t^i      &i=1,\ldots,n \label{p1_2T}\\
 	&x_t^{n+1}=(1+r_{t-1}^{n+1})x_{t-1}^{n+1}   +\sum\limits_{i=1}^{n}(y_t^i - z_t^i - q_t^i) & \label{F3T}\\
	&y_t^i \leq (1+r_{t-1}^{i})x_{t-1}^{i} & i=1,\ldots,n \label{p1_4T}\\
  &x_t^i \leq u_i * \sum\limits_{i=1}^{n+1} (1+r_{t-1}^i) x_{t-1}^i & i=1,\ldots,n \label{p1_5T}\\
	&g_t^i = y_t^i+z_t^i & i=1,\ldots,n \label{F4T}\\
  &(\ell_t^i)^2 \leq 2 s_t^i \frac{q^i_t} {m^i_t}       & i=1,\ldots,n \label{E1T}\\    
	&(w_t^i)^2 \leq 2 v^T_i w_i^T        & i=1,\ldots,n  \label{E2T} \\    
  &\ell_t^i = v_t^i       & i=1,\ldots,n \label{E3T} \\   
	&s_t^i = w_t^i        & i=1,\ldots,n \label{E4T} \\     
	&w_t^i = 0.125        & i=1,\ldots,n \label{E5T} \\    
		&-g_t^i \leq \ell_t^i    & i=1,\ldots,n \label{E21T}\\ 
  &g_t^i  \leq \ell_t^i    & i=1,\ldots,n \label{E22T} \\
	&y_t^i, z_t^i,s_t^i, v_t^i  \ge 0 &         i=1,\ldots,n \label{p1_6T} \\
	&x_t^i \ge 0 &         i=1,\ldots,n+1 \label{p1_LAST} \\
\end{align}
\end{subequations}
with $t=T$.

At stage $t=T-1,\dots,2$, we have
\begin{subequations}\label{MI_t}
\begin{align}
Q_t\left( x_{t-1} \right)=\text{Max} \ & \mathbb{E}\left[ Q_{t+1}\left( x_{t} \right) \right] \label{OBJtt}&\\
{\rm s.t.} \ & \eqref{p1_2T}-\eqref{p1_LAST}  \ . \notag 
\end{align}
\end{subequations}

At the first stage $t=1$, we have:
\begin{subequations}\label{MI_1}
\begin{align}
\text{Max} \  & \mathbb{E}\left[ Q_{2}\left( x_{1} \right) \right] &\\
{\rm s.t.} \ & \eqref{p1_2T}-\eqref{p1_LAST}  \ . & \notag    
\end{align}
\end{subequations}
with $t=1$.

}\fi

\section{Numerical experiments}\label{numsim}


In this section, we evaluate the computational efficiency of the {\sc DDP-REG} and {\sc SDDP-REG} algorithms presented in sections \ref{REDA} and \ref{SREFDA}, and benchmark 
them with standard, non regularized versions of the deterministic and stochastic DDP algorithms. 
The analysis starts (section \ref{simdet}) with the deterministic setting and the {\sc DDP-REG} algorithm tested on a portfolio optimization 
problem with direct transaction costs, and continues in section \ref{STOCHA} with the stochastic case and the {\sc SDDP-REG} algorithm tested on risk-neutral and 
risk-averse formulations involving either direct transaction or market impact costs. In practice, portfolio selection problem parameters (the returns) are not known 
in advance and stochastic optimization models are used for these applications.
We use such models in section \ref{STOCHA}.
However, to compare DDP and DDP-REG, we assume that the parameters of the portfolio problems,
namely the returns, are known over the optimization period. This  
allows us to easily generate feasible problem instances that can be solved with DDP and DDP-REG and to
know what would have been the best return for these instances.

\subsection{Data and parameter settings} \label{DATA}

The problem instances and the algorithms are modelled in Python and the problems are solved with {\sc MOSEK} 8.0.0.50 solver \cite{MOSEK}.
The experiments are carried out using a single thread of an Intel(R) Core(TM) i5-4200M CPU @ 2.50GHz machine.

The following settings are used for the parameters of the portfolio optimization problems described in section \ref{FINM}. 
The budget available is  \$1 billion and can be used to invest in $n=6$ risky securities in addition to cash. The proportional direct transaction costs $\eta = \nu$ are set to 1\%.
The return data of six securities were collected from WRDS \cite{WRDS} for the period ranging from July 2005 to May 2016.
The monthly fixed cash return is equal to $0.2\%$.
The largest position in any security is set to $u_i=20\%$. 
The parameters of the {\sc DDP-REG} and {\sc SDDP-REG} algorithms follow. 
We consider a number $T$ of stages ranging from 10 to 350.
The sample size per stage, i.e., the cardinality of $\Theta_t$ (using the notation of section \ref{SREFDA}), is set to $M=60$.

As we recall from section \ref{convreddp:sec} for DDP-REG and from 
subsection \ref{proxpenaliz} for SDDP-REG,
we need to define sequences $x_{t}^{P,k}$ of prox-centers and $\lambda_{t, k}$ of penalization parameters
to define instances of DDP-REG and SDDP-REG.
In our study, we will use the prox-centers and penalization parameters given in Table \ref{namevariants}
(we recall that no penalization is used for $t=T$ and for $k=1$, i.e., $\lambda_{T,k}=\lambda_{t,1}=0$
for all $t, k$).
This table also contains the names used for the corresponding DDP-REG and SDDP-REG variants.
\begin{table}
\centering
\begin{tabular}{|c||c|c|}
\hline
\begin{tabular}{c}
{\tt{Variant name}} \end{tabular}& \begin{tabular}{l}{\tt{Prox-center }}$x_{t}^{P,k}$\\for $t<T, k>1$\end{tabular} & \begin{tabular}{l}{\tt{Penalization }}$\lambda_{t, k}$\\for $t<T, k>1$\end{tabular}  \\
\hline
\begin{tabular}{l}
DDP-REG-PREV-REG1-$\rho$ or\\SDDP-REG-PREV-REG1-$\rho$
\end{tabular}&   $x_{t-1}^{k}$    &    $\rho^k$ with $0<\rho<1$       \\       
\hline
\begin{tabular}{l}
DDP-REG-PREV-REG2 or\\SDDP-REG-PREV-REG2 \end{tabular}              &           $x_{t-1}^{k}$                  &      $\frac{1}{k^2}$     \\  
\hline
\begin{tabular}{l}
DDP-REG-AVG-REG1-$\rho$ or\\SDDP-REG-AVG-REG1-$\rho$\end{tabular}   &     $\displaystyle  \frac{1}{k-1}\sum_{j=1}^{k-1} x_t^j$               &   $\rho^k$ with $0<\rho<1$ \\
\hline
\begin{tabular}{l}
DDP-REG-AVG-REG2 or\\SDDP-REG-AVG-REG2\end{tabular}  &        $\displaystyle \frac{1}{k-1}\sum_{j=1}^{k-1} x_t^j$                    &        $\frac{1}{k^2}$               \\
\hline
\end{tabular}
\caption{Some variants of DDP-REG and SDDP-REG.}
\label{namevariants}
\end{table}
We recall that in \cite{powellasamov}, only the variant SDDP-REG-PREV-REG1-$\rho$ was tested for linear programs.
In this section, we test all deterministic variants from Table \ref{namevariants} for linear programs and
variant SDDP-REG-PREV-REG2 for multistage stochastic linear and nonlinear programs.

\subsection{Deterministic instances} \label{simdet}

\par In this section, we consider the deterministic counterpart of the portfolio optimization problem with direct transaction costs 
presented in section \ref{linearmodels} using the parameters given in the previous section and 8 different values
for the number $T$ of time periods: $T=10, 50, 100, 150, 200, 250, 300$, and $350$.
We solve these problems using DDP and the following 6 variants of DDP-REG (using the notation of Table \ref{namevariants}):
DDP-REG-PREV-REG1-0.2 (DDP-REG-PREV-REG1-$\rho$ with $\rho=0.2$), DDP-REG-PREV-REG1-0.9 (DDP-REG-PREV-REG1-$\rho$ with $\rho=0.9$),
DDP-REG-PREV-REG2, DDP-REG-AVG-REG1-0.2 (DDP-REG-AVG-REG1-$\rho$ with $\rho=0.2$), DDP-REG-AVG-REG1-0.9 (DDP-REG-AVG-REG1-$\rho$ with $\rho=0.9$),
and DDP-REG-AVG-REG2.\\

\par {\textbf{Stopping criterion.}} When studying the convergence of DDP-REG  in section \ref{REDA}, we have not discussed
the stopping criterion. At each iteration, this algorithm can compute an approximate lower bound on the optimal value
of the problem which is given at iteration $k$ by ${\underline{\mathcal{Q}}}_1^k (x_0 )$ (using the notation 
of section \ref{REDA}), the optimal value of the approximate
problem for the first time period. Observe that we can make of ${\underline{\mathcal{Q}}}_1^k (x_0 )$ an exact lower bound if we take
$\lambda_{1,k}=0$ (such strategy was used in our tests). DDP-REG can also compute at iteration $k$ the  upper bound $\sum_{t=1}^T f_t(x_{t-1}^k , x_t^k )$
on the optimal value. Given a tolerance $\varepsilon$ (taken equal to $10^{-6}$ in our experiments), the 
algorithm stops when the difference between the upper and lower bound is less than $\varepsilon$ (in this case, we have computed
an $\varepsilon$-optimal solution to the problem). Note, however, that since our portfolio problems are maximization problems,
the approximate first stage problem provides an upper bound on the optimal value and $\sum_{t=1}^T f_t(x_{t-1}^k , x_t^k )$
provides a lower bound.\\

\par We have checked that on all instances, all algorithms correctly compute the same optimal value and that the upper and lower
bounds were converging to this optimal value. For illustration, Figure \ref{fig:ddp_non_reg}
displays the evolution of the upper and lower bounds and of the optimality gap across the iterative process
with DDP for the instance with $T=300$.

\begin{figure}[!h]
  \centering
    \includegraphics[height=2.05in,width=0.45\textwidth]{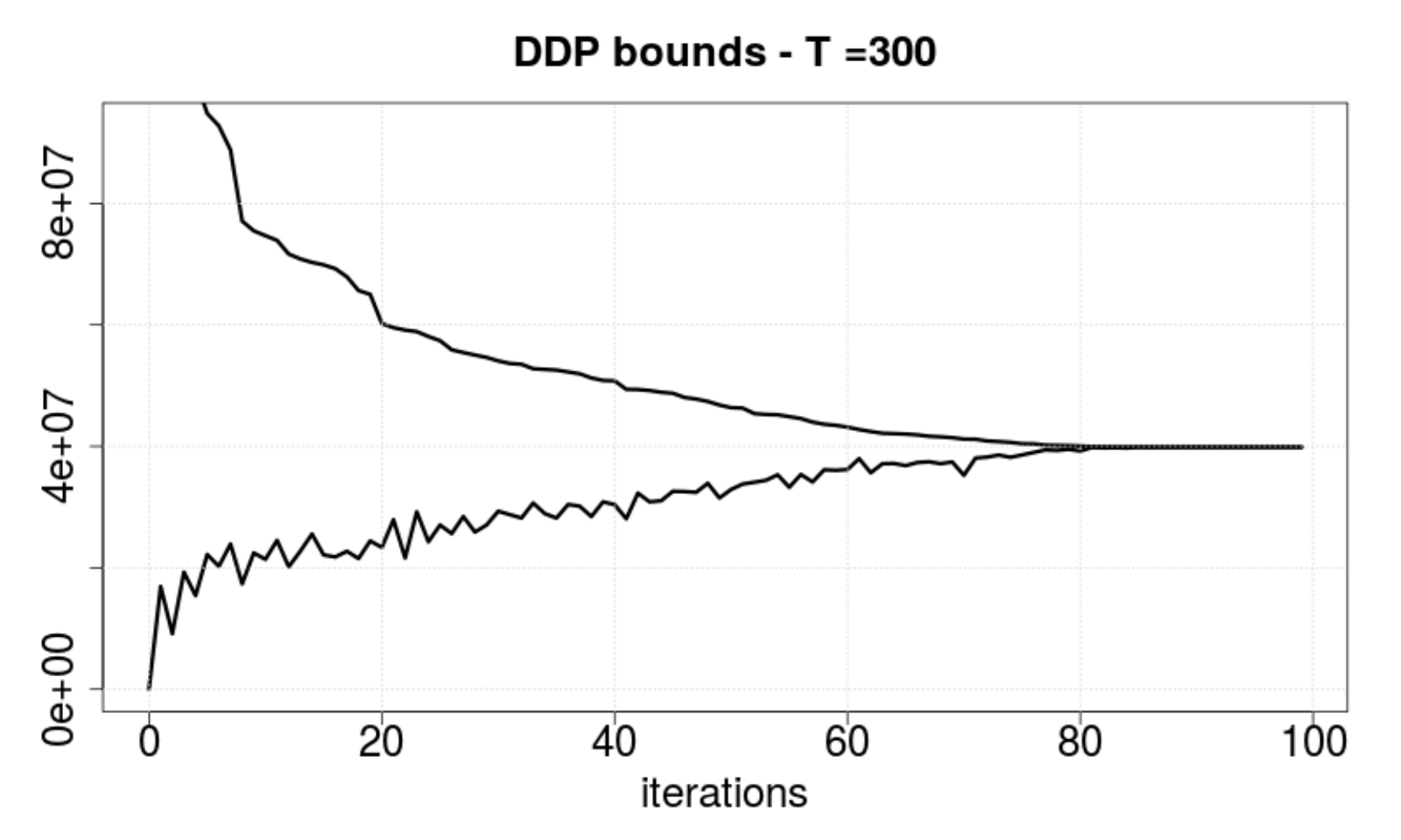}
    \includegraphics[height=2.05in,width=0.45\textwidth]{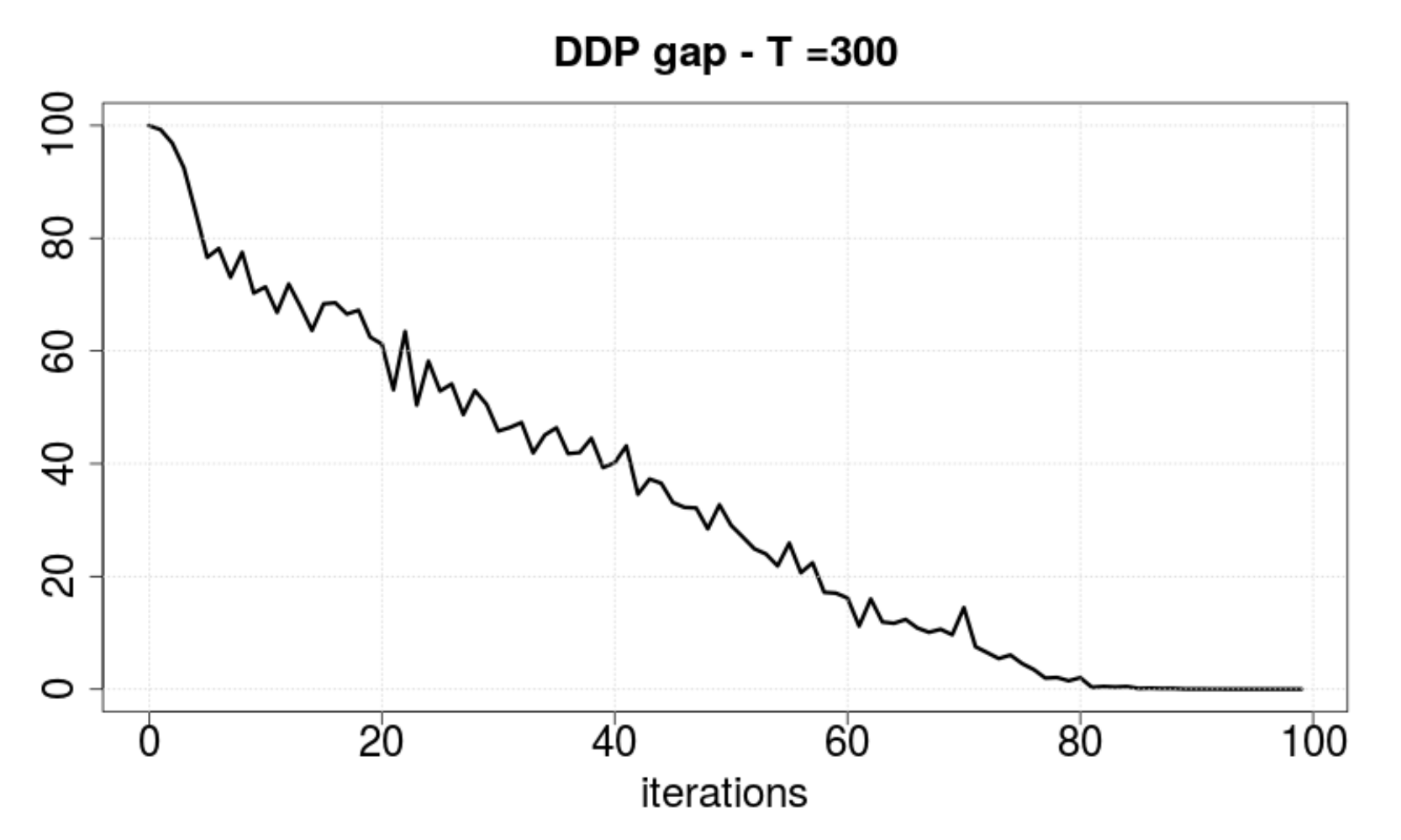}
  \caption{DDP method: DDP lower and upper bounds (left plot) and gap (right plot) in \% of the upper bound for $T=300$.}\label{fig:ddp_non_reg}
\end{figure}

\par The CPU time needed to solve the different instances with DDP and our 6 variants of DDP-REG is given
in Table \ref{tableddpreddpcpu} and the corresponding reduction factor in CPU time
for  these DDP-REG variants is given in Table \ref{tableddpreddpcpufactor}.
The number of iterations of the algorithms is given in Table \ref{tableddpreddpiteration}.
We observe that on all instances DDP-REG variants are much faster and need 
many fewer iterations than DDP.
Most importantly, the benefits of regularization increase as the problem gets larger and the number of stages raises.
When $T$ is large there is a drastic improvement in CPU time with DDP-REG variants.
For instance, for $T=250, 300$, and $350$, the reduction factor in CPU time varies (among the 6 DDP-REG variants) 
respectively in the interval $[80.0,114.3]$, $[71.5,171.6]$, and $[95.5,184.4]$.
Remarkably, the solution time with the regularized algorithm {\sc DDP-REG} is not monotonically increasing with the number of stages, which points out to the 
scalability of the algorithm and the possibility to use it for even larger problems.
As an illustration, the difference in time and number of iterations between DDP and DDP-REG-PREV-REG2 is shown in 
Figure \ref{fig:ddp_comp}, which highlights that the time and iteration differential increase with the number of stages.

\begin{table}
\centering
\begin{tabular}{|c||c|c|c|c|c|c|c|c|}
\hline
$T$ & 10 &50&100&150&200&250&300&350\\
\hline
\hline
DDP &3&69&268&780&1304&2400&4289&5348\\
\hline
DDP-REG-PREV-REG2&1&4&8&13&17&30&25&29\\
\hline
DDP-REG-PREV-REG1-0.2&1&4&12&21&28&23&60&56\\
\hline
DDP-REG-PREV-REG1-0.9&1&4&8&12&17&21&25&29\\
\hline
DDP-REG-AVG-REG2&1&4&8&13&17&30&25&29\\
\hline
DDP-REG-AVG-REG1-0.2&1&5&8&12&17&21&47&55\\
\hline
DDP-REG-AVG-REG1-0.9&1&4&9&13&17&22&26&30\\
\hline
\end{tabular}
\caption{CPU time (in seconds) to solve instances of a portfolio problem of form \eqref{eq:DPP-prob}, namely the deterministic counterpart
of the porfolio models from section \ref{linearmodels}, using DDP and various variants of DDP-REG.}
\label{tableddpreddpcpu}
\end{table}

\begin{table}
\centering
\begin{tabular}{|c||c|c|c|c|c|c|c|c|}
\hline
$T$ & 10 &50&100&150&200&250&300&350\\
\hline
\hline
DDP-REG-PREV-REG2&3.0&17.3&33.5&60.0&76.7&80.0&171.6&184.4\\
\hline
DDP-REG-PREV-REG1-0.2&3.0&17.3&22.3&37.1&46.6&104.4&71.5&95.5\\
\hline
DDP-REG-PREV-REG1-0.9&3.0&17.3&33.5&65.0&76.7&114.3&171.6&184.4\\
\hline
DDP-REG-AVG-REG2&3.0&17.3&33.5&60.0&76.7&80.0&171.6&184.4\\
\hline
DDP-REG-AVG-REG1-0.2&3.0&13.8&33.5&65.0&76.7&114.3&91.3&97.2\\
\hline
DDP-REG-AVG-REG1-0.9&3.0&17.3&29.8&60.0&76.7&109.1&165.0&178.3\\
\hline
\end{tabular}
\caption{CPU time reduction factor for different DDP-REG variants.}
\label{tableddpreddpcpufactor}
\end{table}

\begin{table}
\centering
\begin{tabular}{|c||c|c|c|c|c|c|c|c|}
\hline
$T$ & 10 &50&100&150&200&250&300&350\\
\hline
\hline
DDP&10&26&39&58&66&83&100&104\\
\hline
DDP-REG-PREV-REG2&3&3&3&3&3&4&3&3\\
\hline
DDP-REG-PREV-REG1-0.2&3&3&3&3&3&3&6&5\\
\hline
 DDP-REG-PREV-REG1-0.9&3&3&3&3&3&3&3&3\\
\hline
DDP-REG-AVG-REG2&3&3&3&3&3&4&3&3\\
\hline
DDP-REG-AVG-REG1-0.2&3&3&3&3&3&3&5&5\\
\hline
 DDP-REG-AVG-REG1-0.9&3&3&3&3&3&3&3&3\\
 \hline
\end{tabular}
\caption{Number of iterations to solve instances of a portfolio problem of form \eqref{eq:DPP-prob}, namely the deterministic counterpart
of the porfolio models from section \ref{linearmodels}, using DDP and various variants of DDP-REG.}
\label{tableddpreddpiteration}
\end{table}

 \begin{figure}[!h]
  \centering
    \includegraphics[height=2.05in,width=0.45\textwidth]{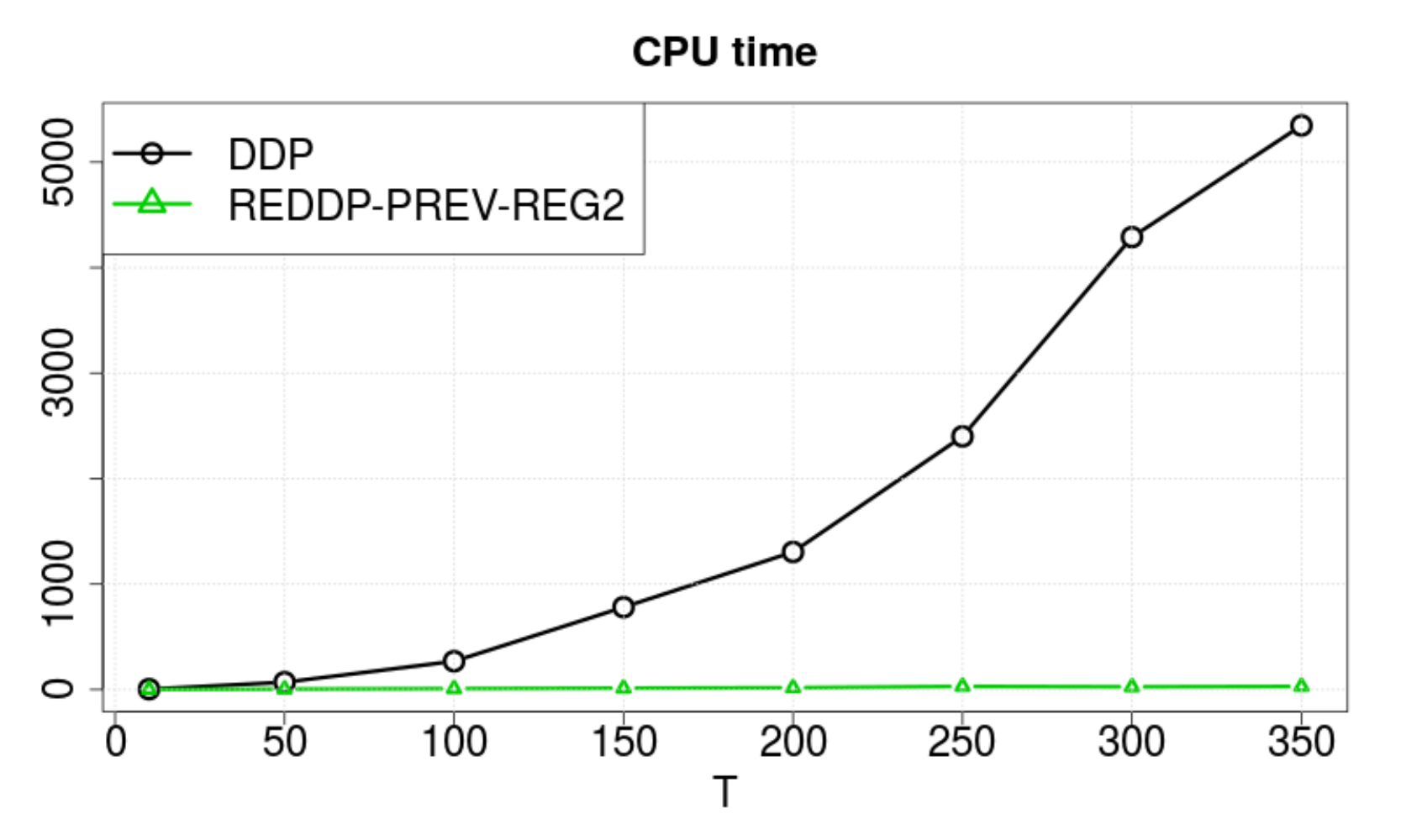}
    \includegraphics[height=2.05in,width=0.45\textwidth]{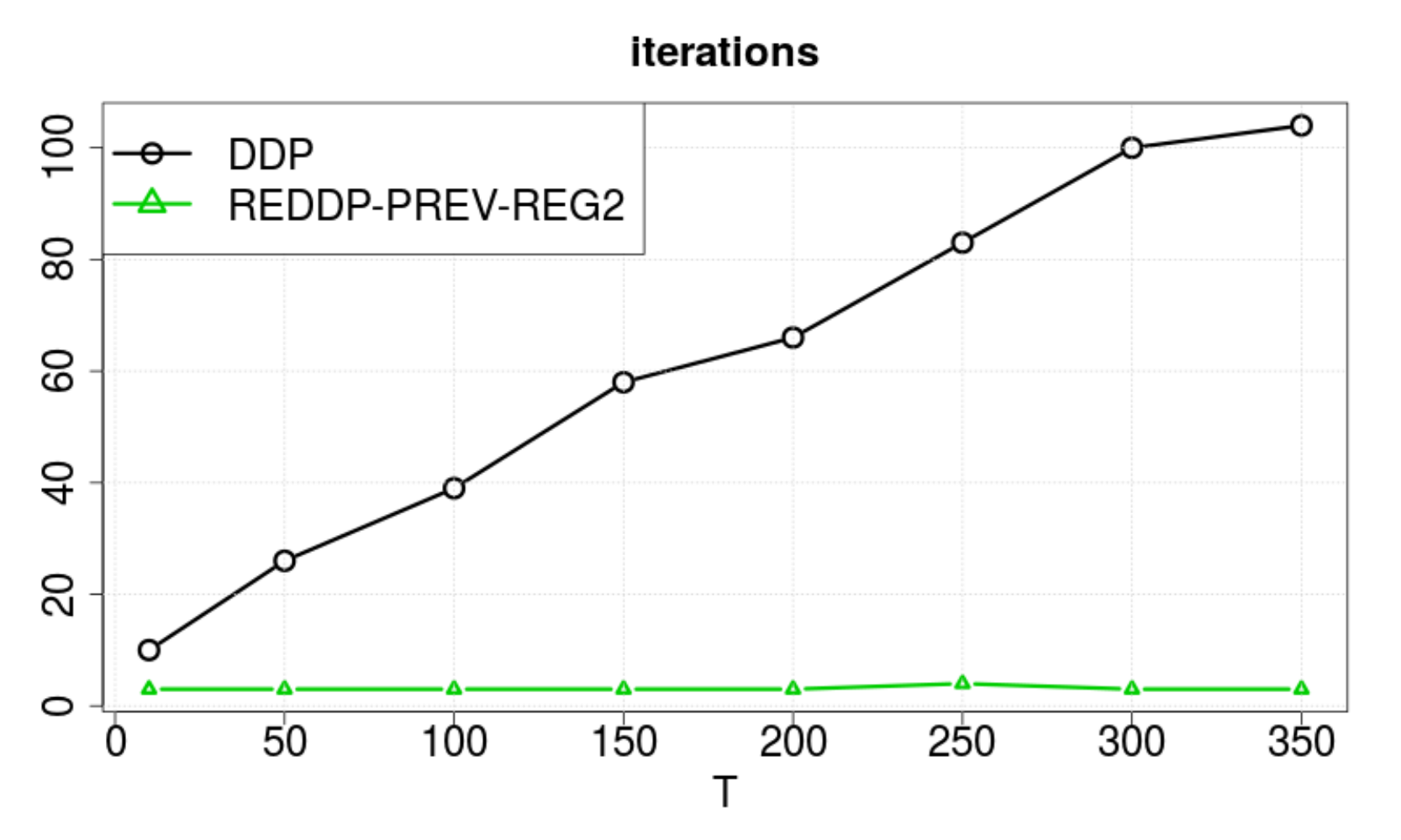}
\caption{Difference in solution time and iteration number between DDP-REG-PREV-REG2 and DDP algorithms.}\label{fig:ddp_comp}
\end{figure}

\subsection{Stochastic instances} \label{STOCHA}

In this section, we evaluate the computational efficiency of the {\sc SDDP-REG} algorithm presented in section \ref{SREFDA}, and benchmark it with the 
standard, non regularized version of the SDDP algorithm. We have implemented the regularization scheme {\sc SDDP-REG-PREV-REG2} given that 
penalization scheme REG2 performed best for the deterministic instances (see section \ref{simdet}). 
The algorithms are tested on three types of problem instances with $T = 12,20,24$ periods: risk-neutral 
portfolio models of subsection \ref{linearmodels}, risk-averse portfolio models of subsection \ref{linearmodels},
and risk-neutral porfolio model with market impact costs from 
subsection \ref{marketimpcostmodel}.\\

\par {\textbf{Stopping criterion.}} For risk-neutral SDDP, we used the following stopping criterion. The algorithm stops if the
gap is $<3\%$.
The gap is defined as $\frac{Ub-Lb}{Ub}$ where $Ub$ and $Lb$ correspond to upper and lower bounds, respectively. 
The upper bound $Ub$ corresponds to the optimal value of the first stage problem (recall that we have a maximization problem), obtained
taking, as for DDP-REG, $\lambda_{1,k}=0$ (if $\lambda_{1,k} \neq 0$, we get a sequence of approximate upper bounds, which, as we have seen,
converges almost surely to the optimal value of the problem). 
The lower bound $Lb$ corresponds to the lower end of a 95\%-one-sided confidence interval on the optimal value for $N=500$ policy realizations, see 
\cite{shapsddp} for a detailed discussion on this stopping criterion. 
Risk averse SDDP was terminated after a fixed number of iterations ($ = 50 $).\\

\subsubsection{Risk-neutral multistage linear problem with direct transaction costs \eqref{eq1dp}-\eqref{eq2dp}}

We report in Table \ref{resultstosddpreg1} the computational time and number of iterations
required for SDDP and SDDP-REG-PREV-REG2 to solve the instance of portfolio problem \eqref{eq1dpra}-\eqref{eq2dpra} obtained taking 
$T = 12,20,24$ and the problem
parameters given in subsection \ref{DATA}. We observe that as in the deterministic case, the regularized
decomposition method converges much faster (it is about twice as fast for $T=24$) and requires
many fewer iterations. We also refer to Figure \ref{boundsriskneutral} where the evolution of the upper and 
lower bounds and the gap (in \% of the upper bound) are represented for SDDP and SDDP-REG-PREV-REG2 for $T=24$. We see
that the gap decreases much faster with SDDP-REG-PREV-REG2.

\begin{table}
\centering
\begin{tabular}{|c|c||c|c|}
\hline
{\tt{T}} & {\tt{Variant}} &   {\tt{CPU time (s)}} & {\tt{Number of iterations}}\\
\hline
\hline
$12$ & SDDP &     14 & 4   \\
\hline
$12$  & SDDP-REG-PREV-REG2 &  6 & 2  \\
  \hline
$20$ & SDDP &   29   &  5  \\
\hline
$20$ & SDDP-REG-PREV-REG2 &  21 &  4 \\
  \hline
$24$ & SDDP &   40   &  6  \\
\hline
$24$  & SDDP-REG-PREV-REG2 & 18  &  3 \\
  \hline
\end{tabular}
\caption{CPU time and number of iterations to solve an instance of a portfolio problem of form \eqref{eq1dp}-\eqref{eq2dp} using SDDP and SDDP-REG-PREV-REG2.}\label{resultstosddpreg1}
\end{table}

\begin{figure}[!h]
  \centering
    \includegraphics[width=0.45\textwidth]{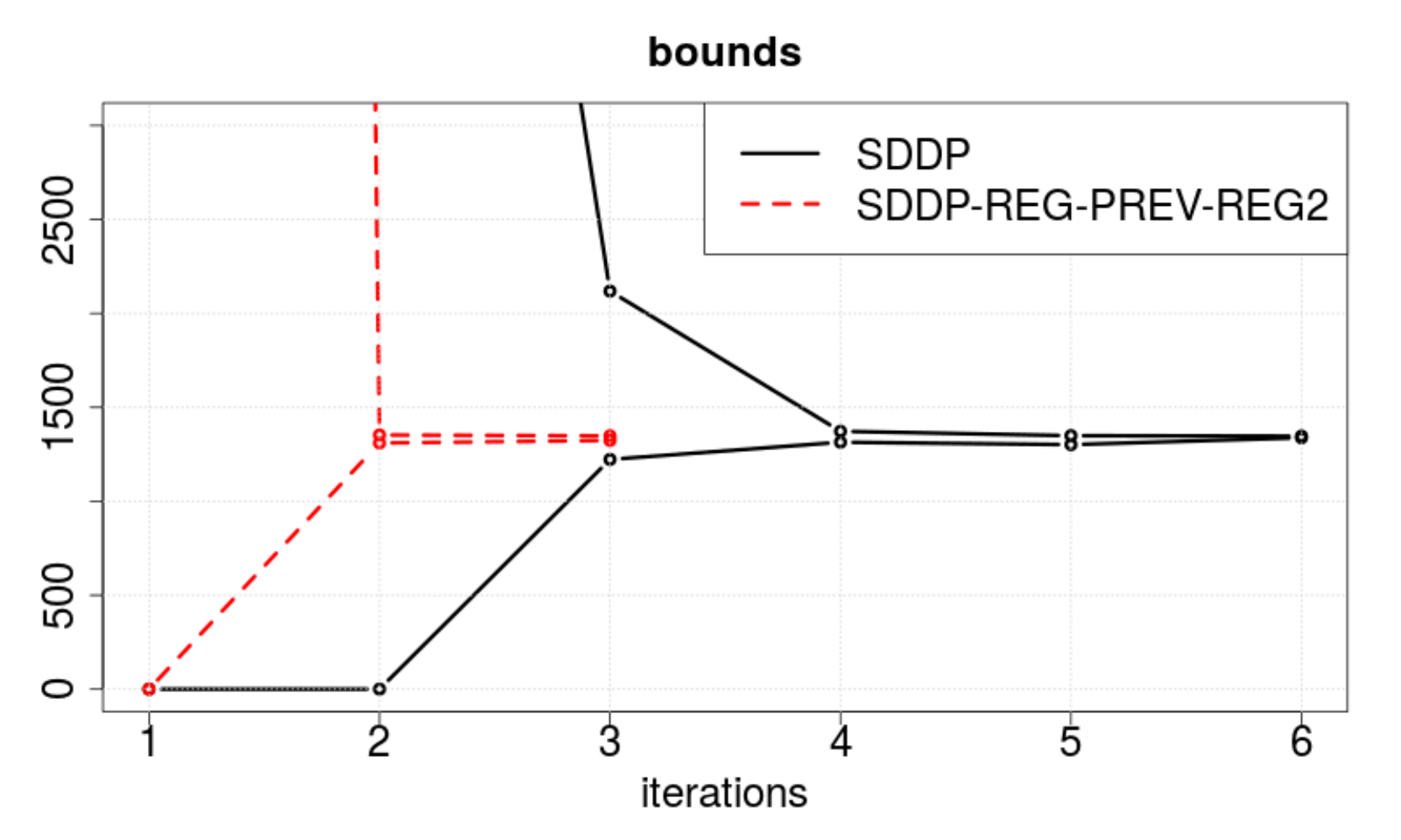}
    \includegraphics[width=0.45\textwidth]{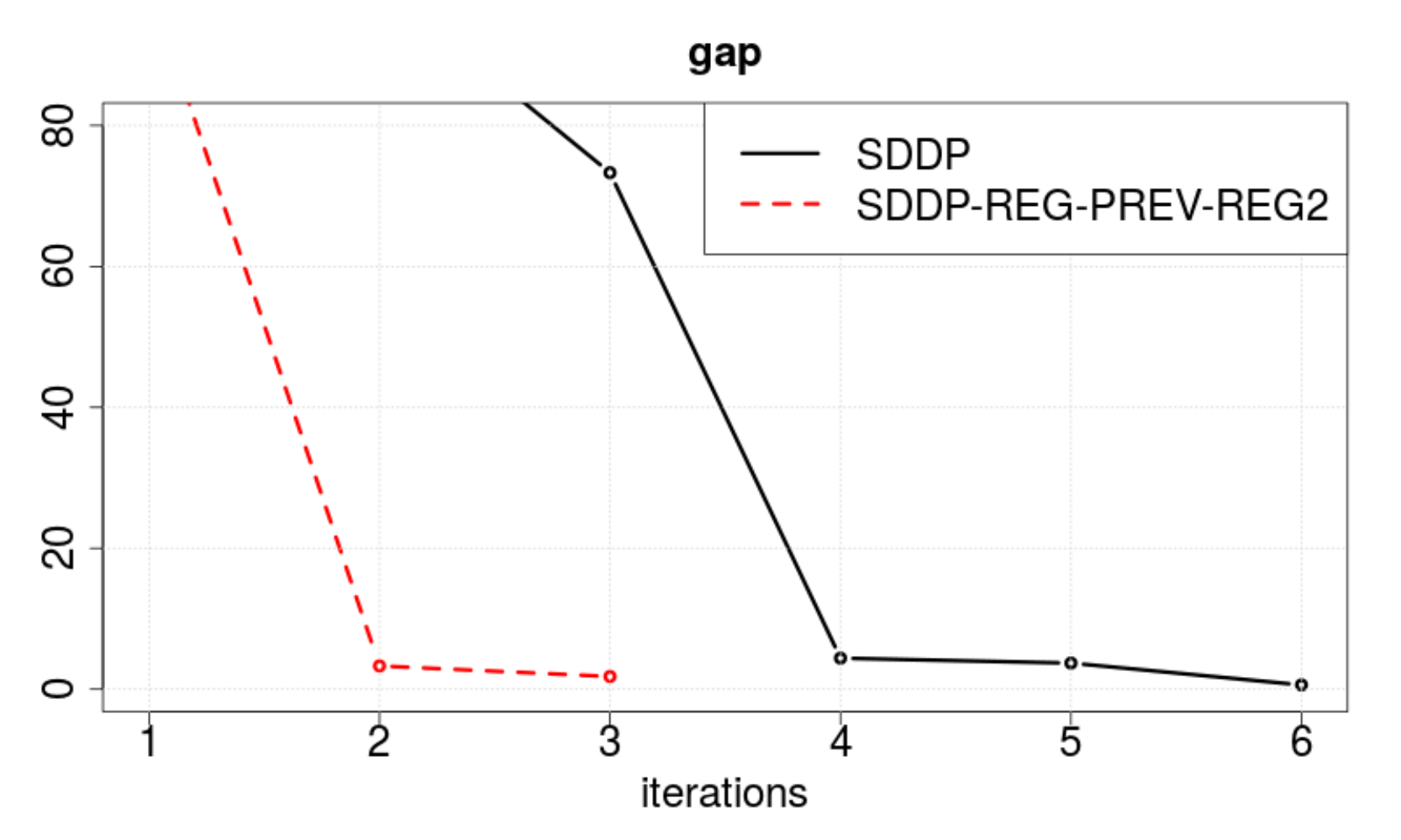}
  \caption{Risk-neutral upper and lower bounds (left plot) and gap (right plot) in \% of the upper bound for $T=24$}\label{boundsriskneutral}
\end{figure}

\subsubsection{Risk-averse multistage linear problem with direct transaction costs \eqref{eq1dpra}-\eqref{eq2dpra}}

We implemented risk-averse models \eqref{eq1dpra}-\eqref{eq2dpra} taking $\kappa_t = 0.1$ and $\alpha_t=0.1$,
running the algorithms for 50 iterations.
The CPU time is reported in Table \ref{resultstosddpreg1ra}. Since both problems are run for the same number
of iterations and since the regularized variant requires solving quadratic problems instead of just linear
programs in the forward passes, it was expected to have a larger computational time with the regularized variant.
However, the difference is small. We also report in Figure \ref{boundsriskneutralra} the evolution of the upper
bounds for SDDP and SDDP-REG-PREV-REG2. We see again
that the upper bound decreases much faster with SDDP-REG-PREV-REG2.

\begin{table}
\centering
\begin{tabular}{|c||c|c|}
\hline
{\tt{Variant}} &   {\tt{CPU time (s)}} & {\tt{Number of iterations}}\\
\hline
\hline
SDDP &     3895 & 50   \\
\hline
SDDP-REG-PREV-REG2 &  3921 & 50  \\
  \hline
\end{tabular}
\caption{CPU time and number of iterations to solve an instance of a portfolio problem of form \eqref{eq1dpra}-\eqref{eq2dpra}
with $T=48$ using SDDP and SDDP-REG-PREV-REG2.}\label{resultstosddpreg1ra}
\end{table}

\begin{figure}[!h]
  \centering
    \includegraphics[width=0.45\textwidth]{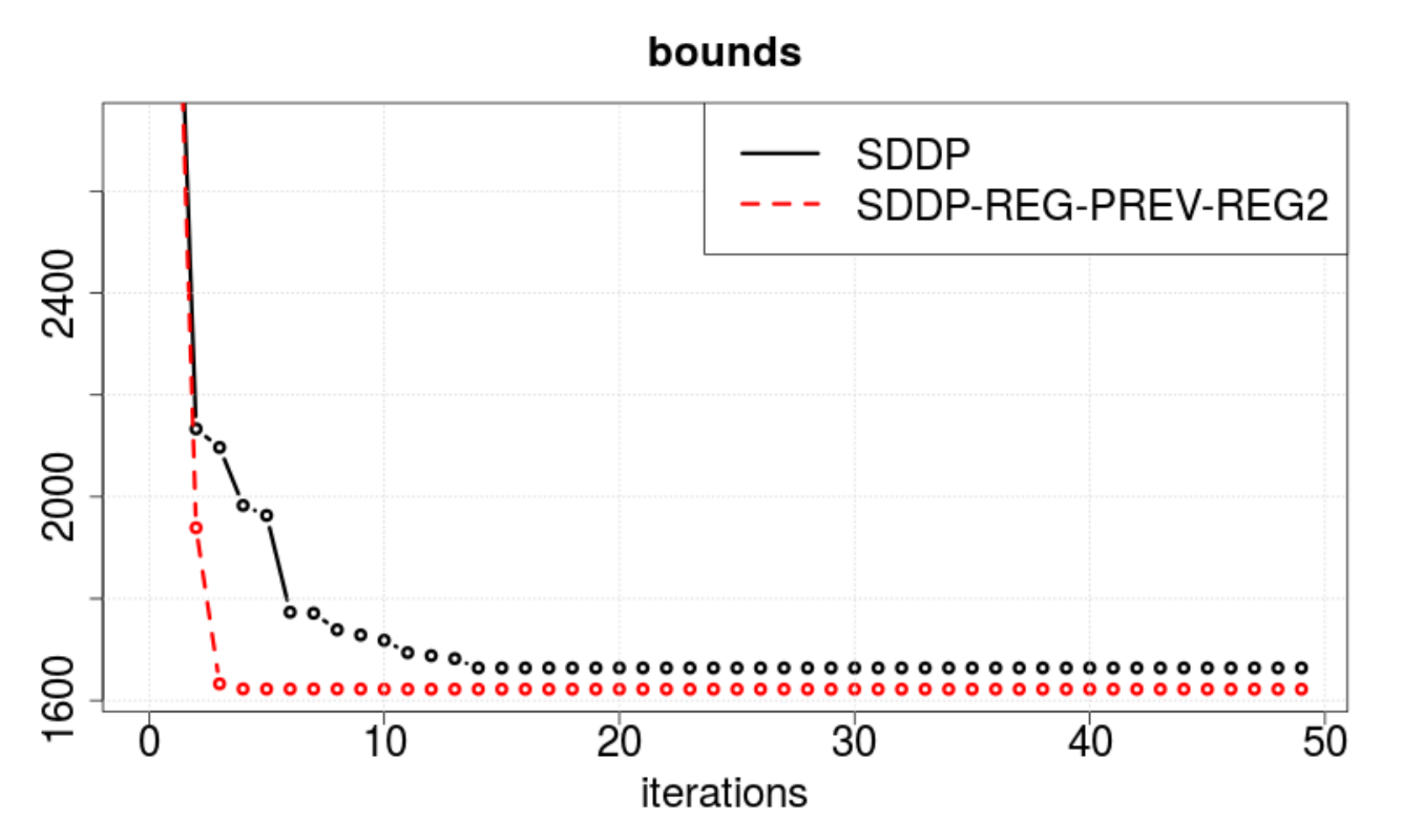}
  \caption{Risk-averse upper bounds, $\kappa_t = 0.1$, $\alpha_t=0.1$.}\label{boundsriskneutralra}
\end{figure}

\subsubsection{Conic risk-neutral multistage stochastic problem with market impact costs from section \ref{marketimpcostmodel}} 

We consider two variants of the portfolio problem with market impact costs given in 
section \ref{marketimpcostmodel}
in which we set the market impact unit cost $m_i, =1,\ldots,n$, to respectively  3 basis points (we recall that a basis point is $0.01\%=10^{-4}$)
for the first model 
and $3\%=0.03$ for the second.
The CPU time and number of iterations to solve these problems with SDDP and SDDP-REG-PREV-REG2
are given in Table \ref{resultstosddpreg1ramc}. The evolution of the upper and lower bounds
and of the gap along the iterations of the algorithms are reported in Figures \ref{boundsriskneutralm1}
and \ref{boundsriskneutralm2} for $T=24$. We observe that when $m_i$ are small the regularized variant is much quicker
and the gap decreases much faster. When $m_i$ increases, in particular for the value $3\%$, more money is invested
in cash and the computational time and gap evolution with the non-regularized and regularized variants of SDDP
are similar.

\begin{table}
\centering
\begin{tabular}{|c|c|c||c|c|}
\hline
$T$ & $m_i$ & {\tt{Variant}}                        &   {\tt{CPU time (s)}} & {\tt{Number of iterations}}\\
\hline
\hline
12& $3bp$  & SDDP &   20   &  8  \\
\hline
12  & $3bp$ & SDDP-REG-PREV-REG2 &    7   &  3  \\
\hline
12 & $3\%$ & SDDP                     &    6  &   3 \\
\hline
12       & $3\%$   & SDDP-REG-PREV-REG2                     &    7  &   3 \\
\hline
20& $3bp$  & SDDP &   43   &  10  \\
\hline
 20  &  $3bp$ & SDDP-REG-PREV-REG2 &   19    &  5  \\
 \hline
20& $3\%$ & SDDP                     &    11  &  3  \\
\hline
 20      & $3\%$  & SDDP-REG-PREV-REG2                     &    11  &  3  \\
\hline
24& $3bp$  & SDDP &    55  &  10  \\
\hline
24                            & $3bp$ & SDDP-REG-PREV-REG2 &    13   &  3  \\
\hline
24 & $3\%$ & SDDP                     &    57  &  11  \\
\hline
 24      & $3\%$   & SDDP-REG-PREV-REG2                     &   13   &   3 \\
\hline
\end{tabular}
\caption{CPU time and number of iterations to solve an instance of a portfolio problem
with market costs (model from section \ref{marketimpcostmodel}) using SDDP and SDDP-REG-PREV-REG2.}\label{resultstosddpreg1ramc}
\end{table}

\begin{figure}[!h]
  \centering
    \includegraphics[width=0.4\textwidth]{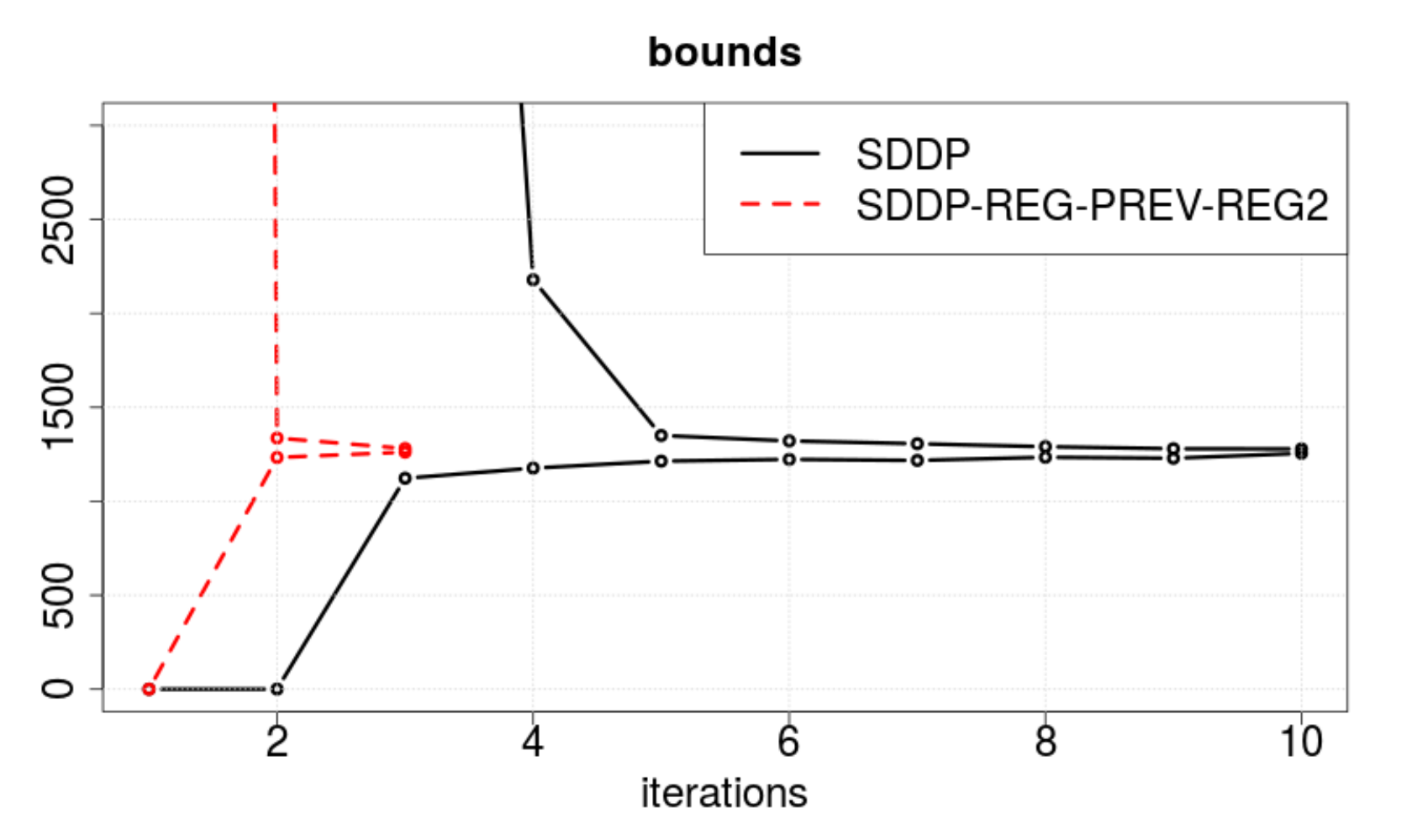}
    \includegraphics[width=0.4\textwidth]{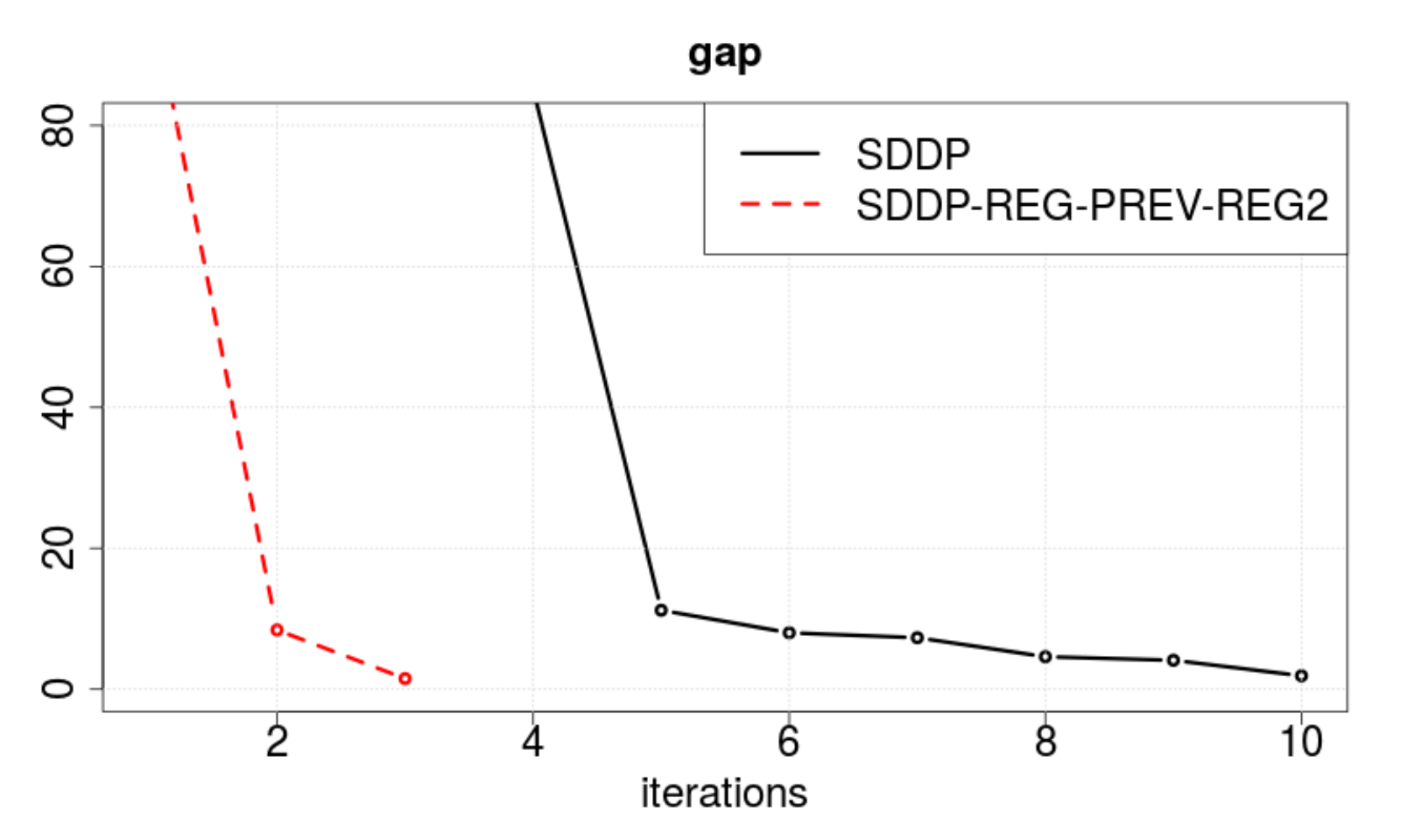}
  \caption{Upper and lower bounds (left plot) and gap (right plot) in \% of the upper bound for the risk-neutral model with market costs, $m_i=3bp$ for $T=24$.}\label{boundsriskneutralm1}
\end{figure}


\begin{figure}[!h]
  \centering
    \includegraphics[width=0.4\textwidth]{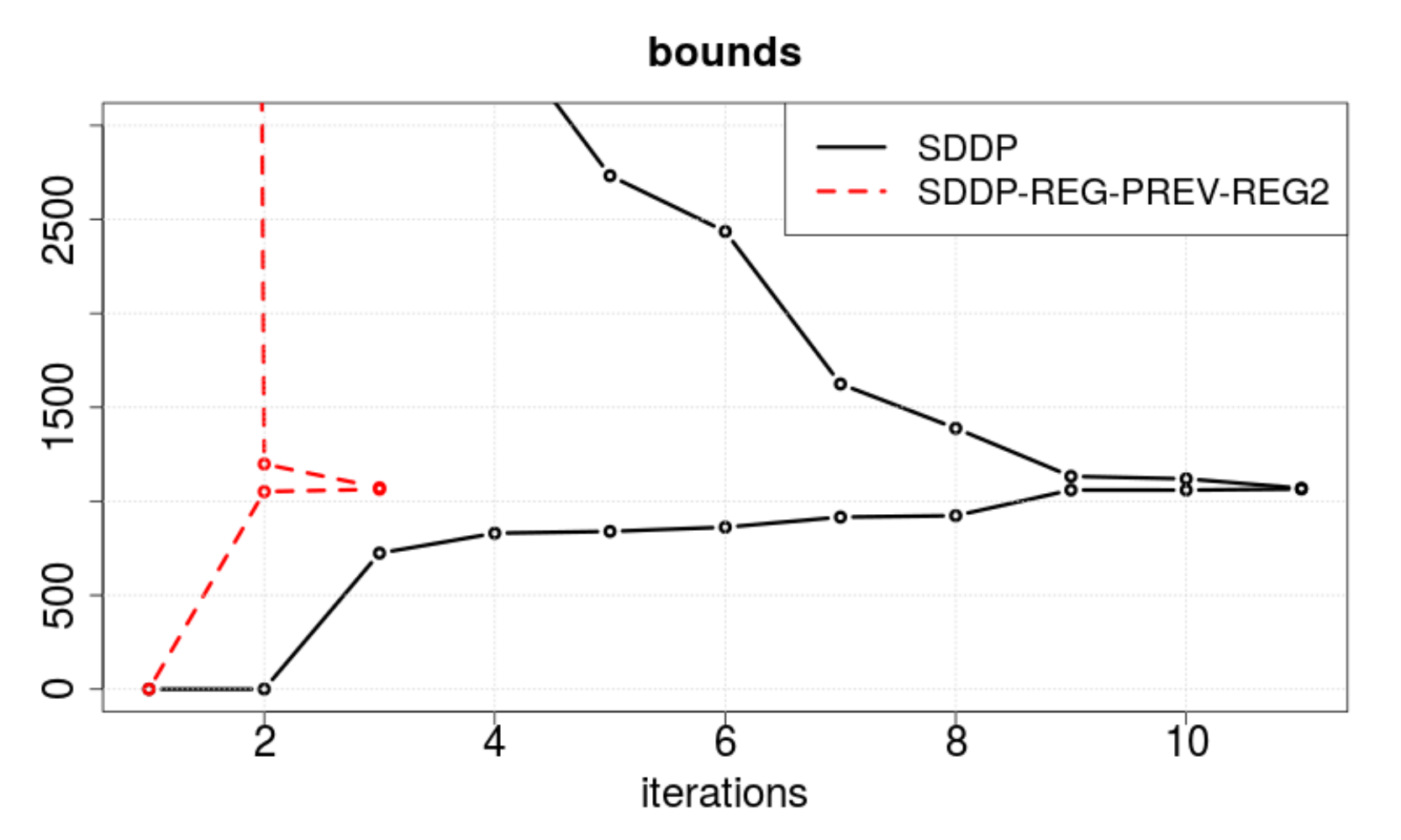}
    \includegraphics[width=0.4\textwidth]{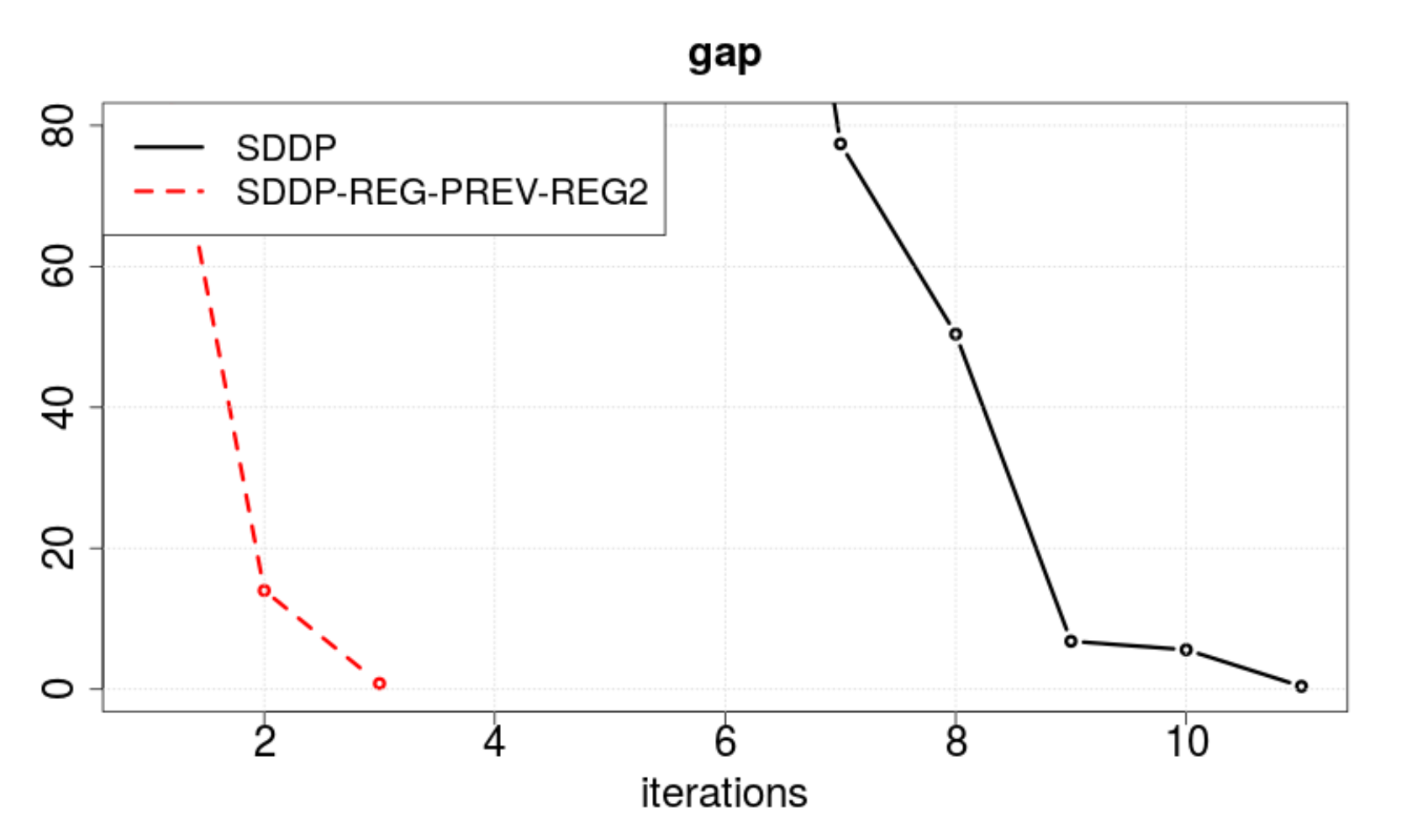}
  \caption{Upper and lower bounds (left plot) and gap (right plot) in \% of the upper bound for the risk-neutral model with market costs, $m_i=3\%$ for $T=24$.}\label{boundsriskneutralm2}
\end{figure}

\section{Conclusion}

We presented and studied regularized variants of DDP and SDDP which are extensions of 
\cite{powellasamov} to nonlinear problems and tested several prox-centers.
On the one hand, for deterministic problems,
the important reduction in CPU time when passing from DDP to DDP-REG,
in the vein of \cite{lem74}, was expected. 
In the stochastic case, we would a priori need different prox-centers for all 
nodes of the scenario tree (see also \cite{senzhou2018}). However,
such regularized variant is not computationally tractable.
Therefore, the proposed SDDP-REG offers a tractable regularized variant
of SDDP whose convergence can be shown for vanishing penalties and which can
converge quicker than SDDP on some problem instances and for 
some choices of prox-centers as shown in our experiments.
We also observe that it is possible, as in \cite{powellasamov}, to partition the decision $x_t$ for stage $t$ into
state $s_{t}$ and control $u_t$ variables and to take 
as prox-centers for stage $t$ and iteration $k$ the state component $s_t^k$
of $x_t^k$. The convergence of both this variant of SDDP-REG as well as the variant that uses prox-centers attached
to nodes of the scenario tree mentioned above can be shown following the steps of the convergence proof
of SDDP-REG given in section \ref{convsreda}.

An interesting topic on the regularization of SDDP is to define a regularization
that can be proved to have better complexity than standard SDDP.

\vspace{-0.1in}
\section*{Acknowledgments} The first author's research was partially supported by an FGV grant, CNPq grant 307287/2013-0, FAPERJ grants E-26/110.313/2014, and E-26/201.599/2014.
This research was initiated during the visit of the second author at FGV. The second author thanks Dr. Guigues for the support provided through FAPERJ grant E-26/201.599/2014.

\vspace{-0.15in}
\bibliographystyle{siamplain}
\bibliography{Biblio}

\section*{Appendix}

\par {\textbf{Proof of Theorem \ref{convalg1}}.} In this proof, all equalities and inequalities hold almost surely.
We show $\mathcal{H}(2)$, $\ldots$, $\mathcal{H}(T+1)$, by induction backwards in time.
$\mathcal{H}(T+1)$ follows from the fact that $\mathcal{Q}_{T+1}=\mathcal{Q}_{T+1}^{k}=0$.
Now assume that $\mathcal{H}(t+1)$ holds for some $t \in \{2,\ldots,T\}$. We want to show that $\mathcal{H}(t)$ holds.
Take a node $n \in {\tt{Nodes}}(t-1)$. Let
$\mathcal{S}_n= \{k \geq 1 :  n_{t-1}^k =n\}$ be the set of iterations 
such that the sampled scenario passes through node $n$. 
Due to Assumption (H3), the set $\mathcal{S}_n$ is infinite.
We first show that
\begin{equation}\label{passesthroughn}
\displaystyle \lim_{k \rightarrow +\infty,\, k \in \mathcal{S}_n} \mathcal{Q}_{t}(x_{n}^{k})-\mathcal{Q}_{t}^{k}(x_{n}^{k} )=0.
\end{equation}
Take $k \in \mathcal{S}_n$. We have $n_{t-1}^k = n, x_{n_{t-1}^k}^k=x_n^k$ and recalling 
\eqref{eqcutctk}, we have $\mathcal{C}_t^k( x_n^k )= \theta_t^k$. Using definition 
\eqref{formulathetak} of $\theta_t^k$, it follows that
\begin{equation}\label{relation1}
\mathcal{Q}_t^k(x_{n}^k) \geq \mathcal{C}_t^k( x_n^k )= \theta_t^k = \displaystyle \sum_{m \in C(n)} p_m {\underline{\mathfrak{Q}}}_t^k (   x_{n}^k , \xi_m  ).
\end{equation}
Now let ${\bar x}_m^k$ such that 
$F_t^{k-1} ( x_{n}^k , {\bar x}_m^k , \xi_m )= {\underline{\mathfrak{Q}}}_t^{k-1}( x_n^k , \xi_m )$
where ${\underline{\mathfrak{Q}}}_t^{k-1}$ is defined by \eqref{defmatfrak} with $k$ replaced by $k-1$.
Using \eqref{relation1} and the definition of $\mathcal{Q}_t$, we get
{\small
\begin{equation}\label{eqfin1conv0}
\begin{array}{lll}
0 \leq \mathcal{Q}_t(x_{n}^k) - \mathcal{Q}_t^k(x_{n}^k) & \leq &\displaystyle \sum_{m \in C(n)} p_m \Big[ \mathfrak{Q}_t(x_{n}^k, \xi_{m})- {\underline{\mathfrak{Q}}}_t^k( x_n^k , \xi_m ) \Big]\\
& \leq &\displaystyle \sum_{m \in C(n)} p_m \Big[ \mathfrak{Q}_t(x_{n}^k, \xi_{m})- {\underline{\mathfrak{Q}}}_t^{k-1}( x_n^k , \xi_m ) \Big]\mbox{ since }{\underline{\mathfrak{Q}}}_t^k \geq {\underline{\mathfrak{Q}}}_t^{k-1}\\
& =  & \displaystyle \sum_{m \in C(n)} p_m \Big[ \mathfrak{Q}_t(x_{n}^k, \xi_{m})- F_t^{k-1} ( x_{n}^k , {\bar x}_m^k , \xi_m ) \Big]\\
&  =   & \displaystyle \sum_{m \in C(n)} p_m \Big[ \mathfrak{Q}_t(x_{n}^k, \xi_{m})- F_t^{k-1} ( x_{n}^k , x_m^k , \xi_m ) \Big]\\
&     &  + \displaystyle \sum_{m \in C(n)} p_m \Big[ F_t^{k-1} ( x_{n}^k, x_m^k , \xi_m ) - F_t^{k-1} ( x_{n}^k , {\bar x}_m^k , \xi_m ) \Big].
\end{array}
\end{equation}
}
Now using the definitions of $F_t^{k-1}$ and $F_t$ we obtain 
\begin{equation}\label{eqfin1conv}
\begin{array}{lll}
\mathfrak{Q}_t(x_{n}^k, \xi_{m})- F_t^{k-1} ( x_{n}^k , x_m^k , \xi_m ) & = & 
\mathfrak{Q}_t(x_{n}^k, \xi_{m})- f_t(x_n^k, x_m^k, \xi_m) - \mathcal{Q}_{t+1}^{k-1}( x_m^k )\\
& = & \mathfrak{Q}_t(x_{n}^k, \xi_{m})- F_t(  x_n^k, x_m^k, \xi_m )\\
&& + \mathcal{Q}_{t+1}( x_m^k ) - \mathcal{Q}_{t+1}^{k-1}( x_m^k ).
\end{array}
\end{equation}
Observing that for every  $m \in C(n)$ the decision $x_m^k \in X_t(x_{n}^k, \xi_m )$, we obtain, using definition \eqref{defmathfrak} of $\mathfrak{Q}_t$, that
$$
F_t(x_{n}^k, x_m^k, \xi_m  )  \geq \mathfrak{Q}_{t}(x_{n}^k, \xi_{m}).
$$
Combining this relation with \eqref{eqfin1conv} gives for $k \in \mathcal{S}_n$
\begin{equation} \label{kinKn}
\begin{array}{l}
\mathfrak{Q}_t(x_{n}^k, \xi_{m})- F_t^{k-1} ( x_{n}^k , x_m^k , \xi_m ) \leq  \mathcal{Q}_{t+1}( x_m^k ) - \mathcal{Q}_{t+1}^{k-1}( x_m^k ).
\end{array}
\end{equation}
Next,
{\small 
\begin{equation} \label{regsddp}
\begin{array}{l}
F_t^{k-1} ( x_{n}^k, x_m^k , \xi_m ) - F_t^{k-1} ( x_{n}^k , {\bar x}_m^k , \xi_m ) \\
=  F_t^{k-1} ( x_{n}^k, x_m^k , \xi_m ) - {\bar F}_t^{k-1} ( x_{n}^k , x_m^k , x_t^{P, k}, \xi_m ) 
+ {\bar F}_t^{k-1} ( x_{n}^k, x_m^k , x_t^{P, k}, \xi_m ) \\
\;\;\; - {\bar F}_t^{k-1} ( x_{n}^k , {\bar x}_m^k , x_t^{P, k}, \xi_m ) + {\bar F}_t^{k-1} ( x_{n}^k , {\bar x}_m^k , x_t^{P, k}, \xi_m ) - F_t^{k-1} ( x_{n}^k , {\bar x}_m^k , \xi_m )\\
 \leq    F_t^{k-1} ( x_{n}^k, x_m^k , \xi_m ) - {\bar F}_t^{k-1} ( x_{n}^k , x_m^k , x_t^{P, k}, \xi_m ) 
+  {\bar F}_t^{k-1} ( x_{n}^k , {\bar x}_m^k , x_t^{P, k}, \xi_m ) \\
\;\;\; - F_t^{k-1} ( x_{n}^k , {\bar x}_m^k , \xi_m ),
\end{array}
\end{equation}
}
where the above inequality comes from the fact ${\bar x}_m^k \in X_t( x_{n}^k, \xi_m)$, i.e., ${\bar x}_m^k$ is feasible for optimization problem
\eqref{defxtkj} with objective function ${\bar F}_t^{k-1} ( x_{n}^k, \cdot , x_t^{P, k}, \xi_m )$ and optimal solution $x_m^k$.
We get
{\small
\begin{equation} \label{regsddp1}
\begin{array}{lll}
0 \leq F_t^{k-1} ( x_{n}^k, x_m^k , \xi_m ) - F_t^{k-1} ( x_{n}^k , {\bar x}_m^k , \xi_m ) & \leq  & 
\lambda_{t, k}( \|{\bar x}_m^k -  x_t^{P, k}\|^2 -  \| x_m^k -  x_t^{P, k}\|^2  )\\
& \leq &  \lambda_{t, k} \|{\bar x}_m^k -  x_t^{P, k}\|^2 \leq \lambda_{t, k} D( \mathcal{X}_t)^2,
\end{array}
\end{equation}
}
where $D( \mathcal{X}_t)$ is the diameter of $\mathcal{X}_t$ (finite, since $\mathcal{X}_t$ is compact).
Plugging \eqref{regsddp1} and \eqref{kinKn} into \eqref{eqfin1conv0} yields for any $k \in \mathcal{S}_n$
\begin{equation}\label{eqfin1conv2}
0 \leq \mathcal{Q}_t(x_{n}^k) - \mathcal{Q}_t^k(x_{n}^k)  \leq \lambda_{t, k} D( \mathcal{X}_t)^2 +
\displaystyle \sum_{m \in C(n)} p_m \Big( \mathcal{Q}_{t+1}( x_m^k ) - \mathcal{Q}_{t+1}^{k-1}( x_m^k ) \Big).
\end{equation}
Using the induction hypothesis $\mathcal{H}(t+1)$, we have for every child node $m$ of node $n$:
\begin{equation}\label{inductionm}
\displaystyle \lim_{k \rightarrow +\infty} \mathcal{Q}_{t+1}(x_{m}^{k})-\mathcal{Q}_{t+1}^{k}(x_{m}^{k} )=0.
\end{equation}
Now recall that $\mathcal{Q}_{t+1}$ is convex on the compact set $\mathcal{X}_t$ (Proposition \ref{convexityrec}),
$x_{m}^k \in  \mathcal{X}_t$ for every child node $m$ of node $n$,
and the functions $\mathcal{Q}_{t+1}^{k}, k \geq 1$, are Lipschitz continuous 
with $\mathcal{Q}_{t+1} \geq \mathcal{Q}_{t+1}^{k} \geq \mathcal{Q}_{t+1}^{k-1}$ on $\mathcal{X}_t$
(Lemma \ref{lipqtk}).
It follows that we can use Lemma A.1 in \cite{lecphilgirar12} to deduce from \eqref{inductionm} that
for every $m \in C(n)$
$$
\displaystyle \lim_{k \rightarrow +\infty   } \mathcal{Q}_{t+1}(x_{m}^{k})-\mathcal{Q}_{t+1}^{k-1}(x_{m}^{k} )=0.
$$
Combining this relation with \eqref{eqfin1conv2} and using the fact that $\lim_{k \rightarrow +\infty} \lambda_{t, k}=0$, we obtain 
\begin{equation}\label{convfirst}
\lim_{k \rightarrow +\infty, k \in \mathcal{S}_n    } \mathcal{Q}_{t}(x_{n}^{k})-\mathcal{Q}_{t}^{k}(x_{n}^{k} )=0.
\end{equation}
To show $\mathcal{H}(t)$, it remains to show that 
\begin{equation}\label{inductionm3}
\displaystyle \lim_{k \rightarrow +\infty, k \notin \mathcal{S}_n    } \mathcal{Q}_{t}(x_{n}^{k})-\mathcal{Q}_{t}^{k}(x_{n}^{k} )=0.
\end{equation}
Relation \eqref{inductionm3} can be shown following the end of the proof of 
Theorem 4.1 in \cite{guigues2015siopt}, by contradiction and using the Strong Law of Large Numbers (the same arguments
were first used in a similar context in Theorem 3.1 of \cite{lecphilgirar12} but for a different problem formulation
and sampling scheme). The key to the proof being the fact that the sampled nodes for iteration $k$ are independent on the decisions
computed at the nodes of the scenario tree for  that iteration and on recourse functions $\mathcal{Q}_{t+1}^{k-1}$.
This achieves the proof of (i).
\par (ii) Recalling that the root node $n_0$ with decision 
$x_0$ taken at that node has a single child node $n_1$ with corresponding decision 
$x_{n_1}^k$ at iteration $k$, the computations in (i) show that for every $k \geq 1$\footnote{Though when deriving these relations
in (i) we had fixed $k \in \mathcal{S}_n$, the inequalities we now re-use for (ii) are valid for any $k \geq 1$.}, we have
\begin{equation}\label{convoptvalfirst}
\begin{array}{lll}
0 \leq  \mathfrak{Q}_1( x_0 , \xi_1 ) - {\underline{\mathfrak{Q}}}_1^k(x_0, \xi_1 )  & \leq &  \mathfrak{Q}_1( x_0 , \xi_1 ) - F_1^{k-1}( x_0, x_{n_1}^k, \xi_1 ) + \lambda_{1, k} D(\mathcal{X}_1)^2,\\
& \leq & \mathcal{Q}_2( x_{n_1}^k ) - \mathcal{Q}_2^{k-1}(  x_{n_1}^k ) + \lambda_{1, k} D(\mathcal{X}_1)^2. 
\end{array}
\end{equation}
We have shown in (i) that $\lim_{k \rightarrow +\infty} \mathcal{Q}_2( x_{n_1}^k ) - \mathcal{Q}_2^{k-1}(  x_{n_1}^k )=0$. Plugging this relation
into \eqref{convoptvalfirst} shows that 
{\small
$$
\lim_{k \rightarrow +\infty} {\underline{\mathfrak{Q}}}_1^k(x_0, \xi_1 ) = \lim_{k \rightarrow +\infty} F_1^{k-1}( x_0, x_{n_1}^k, \xi_1 ) =  \lim_{k \rightarrow +\infty} {\bar F}_1^{k-1}(x_0, x_{n_1}^k , x_1^{P, k}, \xi_1)=\mathfrak{Q}_1( x_0 , \xi_1 ).
$$
}
Now take an accumulation point $(x_n^*)_{n \in \mathcal{N}}$ 
of the sequence $((x_n^k)_{n \in \mathcal{N}})_{k \geq 1}$ and let
$K$ be an infinite set of iterations such that for every $n \in \mathcal{N}$, $\lim_{k \rightarrow +\infty, k \in K} x_n^k= x_n^*$.\footnote{The existence of an accumulation point comes from the fact that the decisions belong almost surely to a compact set.}
Using once again computations from (i), we get for any $k \geq 1$, $t=1,\ldots,T$,
$n \in {\tt{Nodes}}(t-1)$, $m \in C(n)$,
{\small
$$
\begin{array}{lll}
0 \leq  \mathfrak{Q}_t( x_n^k , \xi_m ) - {\underline{\mathfrak{Q}}}_t^{k-1}(x_n^k , \xi_m )  & \leq &  \mathfrak{Q}_t( x_n^k , \xi_m ) - F_t^{k-1}( x_n^k , x_{m}^k, \xi_m ) + \lambda_{t, k} D(\mathcal{X}_t)^2,\\
& \leq & \mathcal{Q}_{t+1}( x_{m}^k ) - \mathcal{Q}_{t+1}^{k-1}(  x_{m}^k ) + \lambda_{t, k} D(\mathcal{X}_t)^2, 
\end{array}
$$
}
which can be written 
$$
-\lambda_{t, k} D(\mathcal{X}_t)^2 \leq \mathfrak{Q}_t( x_n^k , \xi_m ) - F_t^{k-1}( x_n^k , x_{m}^k, \xi_m )  \leq \mathcal{Q}_{t+1}( x_{m}^k ) - \mathcal{Q}_{t+1}^{k-1}(  x_{m}^k ).
$$
Since $\lim_{k \rightarrow +\infty} \mathcal{Q}_{t+1}( x_{m}^k ) - \mathcal{Q}_{t+1}^{k-1}(  x_{m}^k )=0$ (due to (i)), the above
relation shows that 
\begin{equation}\label{convoptval}
\lim_{k \rightarrow +\infty} \mathfrak{Q}_t( x_n^k , \xi_m ) - F_t^{k-1}( x_n^k , x_{m}^k, \xi_m )=0.
\end{equation}
We will now use the continuity of $\mathfrak{Q}_t(\cdot, \xi_m)$ which follows from (H2) (see Lemma 3.2 in \cite{guigues2015siopt} for a proof).
We have
\begin{equation}\label{optsolconv}
\begin{array}{lll}
\mathfrak{Q}_{t} ( x_{n}^* , \xi_m )  & = & \displaystyle  \lim_{k \rightarrow +\infty, k\in K} \mathfrak{Q}_{t} ( x_{n}^k , \xi_m ) \mbox{ using the continuity of }\mathfrak{Q}_t(\cdot, \xi_m),\\
& =& \displaystyle \lim_{k \rightarrow +\infty, k\in K} F_t^{k-1}( x_n^k , x_{m}^k, \xi_m ) \mbox{ using }\eqref{convoptval},\\
& = & \displaystyle \lim_{k \rightarrow +\infty, k\in K} f_t(x_n^k , x_{m}^k, \xi_m ) + \mathcal{Q}_{t+1}^{k-1}(x_m^k ),\\
& \geq & f_t(x_n^*, x_m^*, \xi_m) + \displaystyle \lim_{k \rightarrow +\infty, k\in K} \mathcal{Q}_{t+1}(x_m^k ) \mbox{ using (i) and lsc of }f_t,\\
& \geq &f_t(x_n^*, x_m^*, \xi_m) + \mathcal{Q}_{t+1}(x_m^* )=F_t(x_n^*, x_m^*, \xi_m)
\end{array}
\end{equation}
where for the last inequality we have used the continuity of $\mathcal{Q}_{t+1}$.
To prove (ii) it suffices to observe that the sequence $(x_n^k, x_m^k)_{k \in K}$ belongs to the set 
$$
{\bar X}_{t, m}=\{(x_{t-1},x_t) \in \mathcal{X}_{t-1}  \small{\times} \mathcal{X}_t : g_t(x_{t-1},x_t,\xi_m ) \leq 0, \,A_m x_t + B_m x_{t-1} = b_m\}
$$
and this set is closed since $g_t$ is lower semicontinuous and $\mathcal{X}_t$ is closed.
Thus, $x_m^* \in X_t(x_n^*, \xi_m)$, which, together with \eqref{optsolconv}, shows that
$x_m^*$ is an optimal solution of $\mathfrak{Q}_{t} ( x_{n}^* , \xi_m )=\inf\{F_t(x_n^*, x_m, \xi_m) \;:\;x_m \in X_t(x_n^*, \xi_m)\}$
and
achieves the proof of~(ii).

\end{document}